\documentclass[11pt, reqno ]{amsart}

\usepackage{graphicx}              
\usepackage{amsmath}               
\usepackage{amsfonts}              
\usepackage{amssymb}
\usepackage{color}
 \usepackage{float}
 
\newtheorem{thm}{Theorem}[section]
\newtheorem{lem}[thm]{Lemma}
\newtheorem{prop}[thm]{Proposition}
\newtheorem{cor}[thm]{Corollary}

\newtheorem{obs}[thm]{Observation}

\numberwithin{equation}{section}

\newcommand{\bd}[1]{\mathbf{#1}}  
\newcommand{\RR}{\mathbb{R}}      
\newcommand{\ZZ}{\mathbb{Z}}      
\newcommand{\CC}{\mathbb{C}}
\newcommand{\mat}[1]{\left(\begin{matrix} #1 \end{matrix} \right)}  
\newcommand{\al}[1]{\begin{align}#1\end{align}}
\newcommand{\aln}[1]{\begin{align*}#1\end{align*}}
\newcommand{\pp}{\mathcal{P}}
\newcommand{\KK}{\mathcal{K}}
\newcommand{\ii}{\bd{i}}
\newcommand{\DD}{\mathbb{D}}
\newcommand{\PP}{\mathcal{P}}

\newcommand{\bt}{\bigtriangleup}

\newcommand{\mi}{\mathcal{I}}

\newcommand{\area}{\operatorname{area}}

\newcommand{\ave}[1]{\left\langle#1\right\rangle} 

\newcommand{\Hecke}{\mathbb G}

\begin{document}

\nocite{*}

\title{Gap Distributions in Circle Packings}

\author{Zeev Rudnick and Xin Zhang}
\address{Raymond and Beverly Sackler School of Mathematical Sciences,
Tel Aviv University, Tel Aviv 69978, Israel}
\email{rudnick@post.tau.ac.il}
\email{xz87@illinois.edu}
\thanks{The research leading to these results has received funding from the European Research Council   under the European Union's Seventh Framework Programme (FP7/2007-2013) / ERC grant agreement
n$^{\text{o}}$ 320755}
\date{\today}

\begin{abstract}
We determine the distribution of nearest neighbour spacings between the tangencies to a fixed circle in a class of circle packings generated by reflections. 
We use a combination of geometric tools and the theory of automorphic forms. 
\end{abstract}
\maketitle


\section{Introduction}
 
\subsection{A class of circle packings}

In recent years, there has been increasing interest in the quantitative study of aspects of circle packings, in particular for Apollonian packings. For instance, the asymptotic count of the number of circles with bounded curvature has been determined \cite{KO, OS}, 
as well  as several  questions concerning the arithmetic of Apollonian packings, see   \cite{Sarnak, Fuchs}. 
 
In this paper, we investigate local statistics for the distribution
of tangencies to a fixed circle in a circle packing,  in a  class of packings generated by reflections.  
The packings that we study are formed by starting 
with an initial finite configuration $\mathcal K$ of circles in the plane with disjoint interiors,
and such that the gaps between circles are curvilinear triangles, as in Figure~\ref{fig configurations}. 
We then form the group $\mathfrak S$ generated by reflections in the dual circles defined by the triangular gaps
of the configuration, see \S~\ref{sec:Construct}, and by applying the elements of $\mathfrak S$ to the initial configuration $\mathcal K$,
we obtain an infinite circle  packing  $\PP$, see Figure~\ref {fig configurations} for an example. 

\begin{figure}[h]
\centering
\begin{minipage}{.4\textwidth}
  \centering
  \includegraphics[width=4cm]{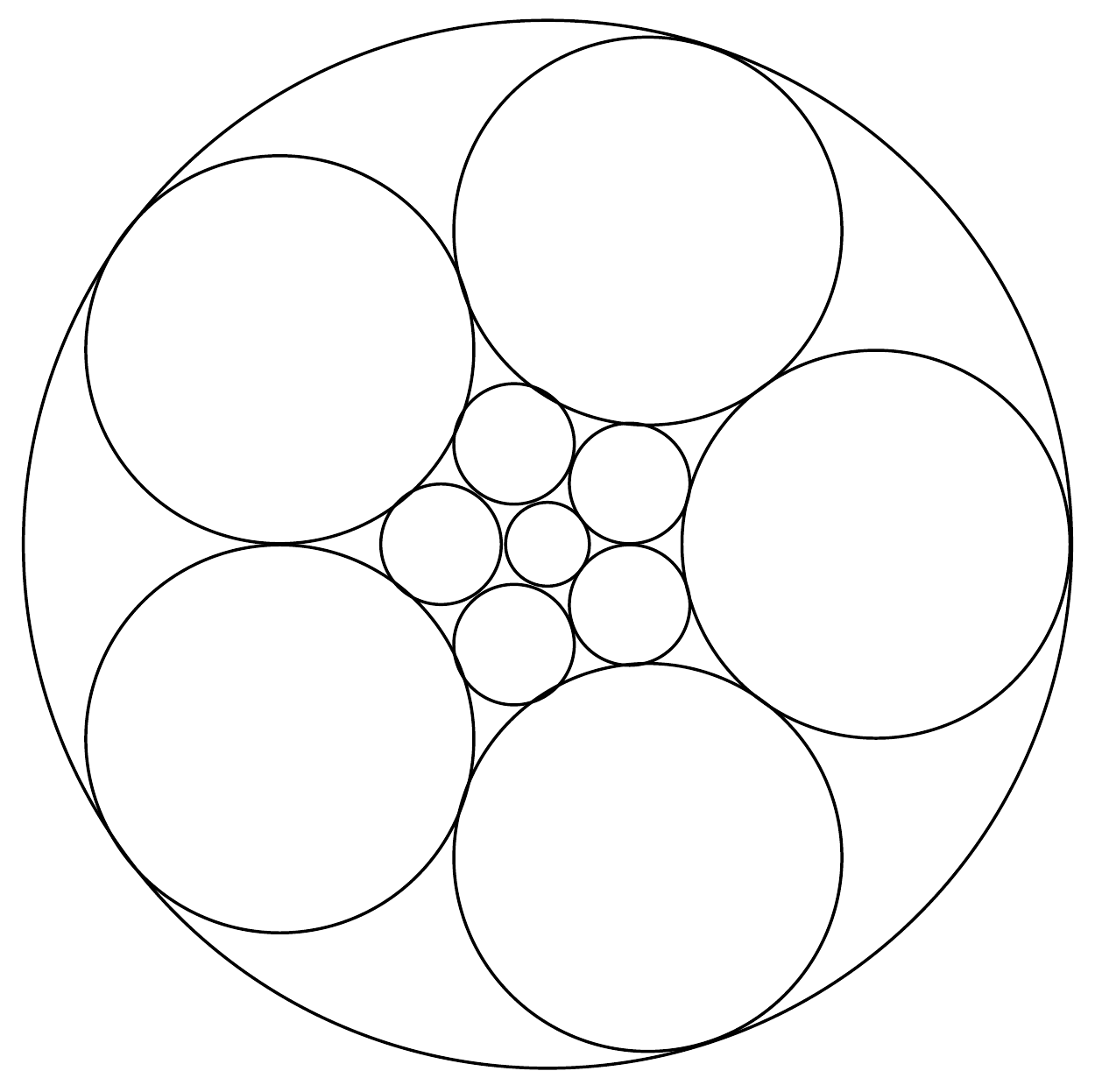}
\end{minipage}%
\begin{minipage}{.6\textwidth}
  \centering
  \includegraphics[width=6cm]{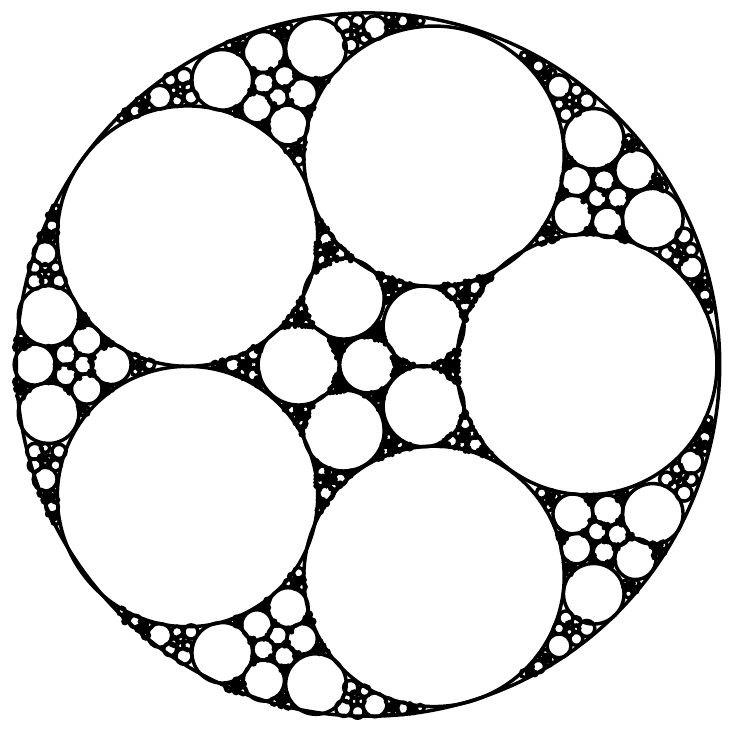}
\end{minipage}
\caption{An Apollonian-9 packing found by Butler, Graham, Guettler and 
Mallows. 
On the left is the initial configuration, for which the tangency graph is the icosahedron. }
 \label{fig configurations}
\end{figure}

We create such initial configurations by starting with a finite triangulation  $G$ of the Riemann sphere $\widehat \CC = \CC\cup \infty$. By the circle packing theorem of Koebe, Andreev and Thurston \cite{StephensonNotices}, there is a circle packing $\mathcal K$ of $\widehat \CC$ having $G$ as its tangency graph, that is a collection of disks with disjoint interiors on $\widehat \CC$ corresponding to the vertices of $G$ 
where two disks are tangent 
if and only if the corresponding vertices are connected in $G$. The gaps between the disks are the connected components of the complement $\widehat \CC\backslash \cup_{C\in \mathcal K}C$ 
of the disks in  $\widehat \CC$. They are triangular 
because we assume that $G$ is a triangulation, that is each face of $G$ is a triangle.  
After stereographic projection from the point $\infty$, we obtain a circle packing $\mathcal{K}$ in the finite plane $\CC$. 

\subsection{Tangencies} 

We fix a base circle $C_0\in \PP$ and consider the  subset $\PP_0\subset \PP$ of circles
tangent to $C_0$. Let $\mathcal{I}\subseteq C_0$ be an arc (if $C_0$ is a line, take $\mathcal{I}$ to be a bounded interval).  Let $\mathcal{A}_{\mathcal{I}}$ be the set of tangencies in $\mathcal{I}$, and $\mathcal{A}_{T,\mathcal{I}}$ be the subset of $\mathcal{A}_{\mathcal{I}}$ whose corresponding circles in $\pp_0$ have curvatures bounded by $T$, see Figure~\ref{fig tangencies}.  

\begin{figure} [b]
\begin{center}
\includegraphics[width=5cm]{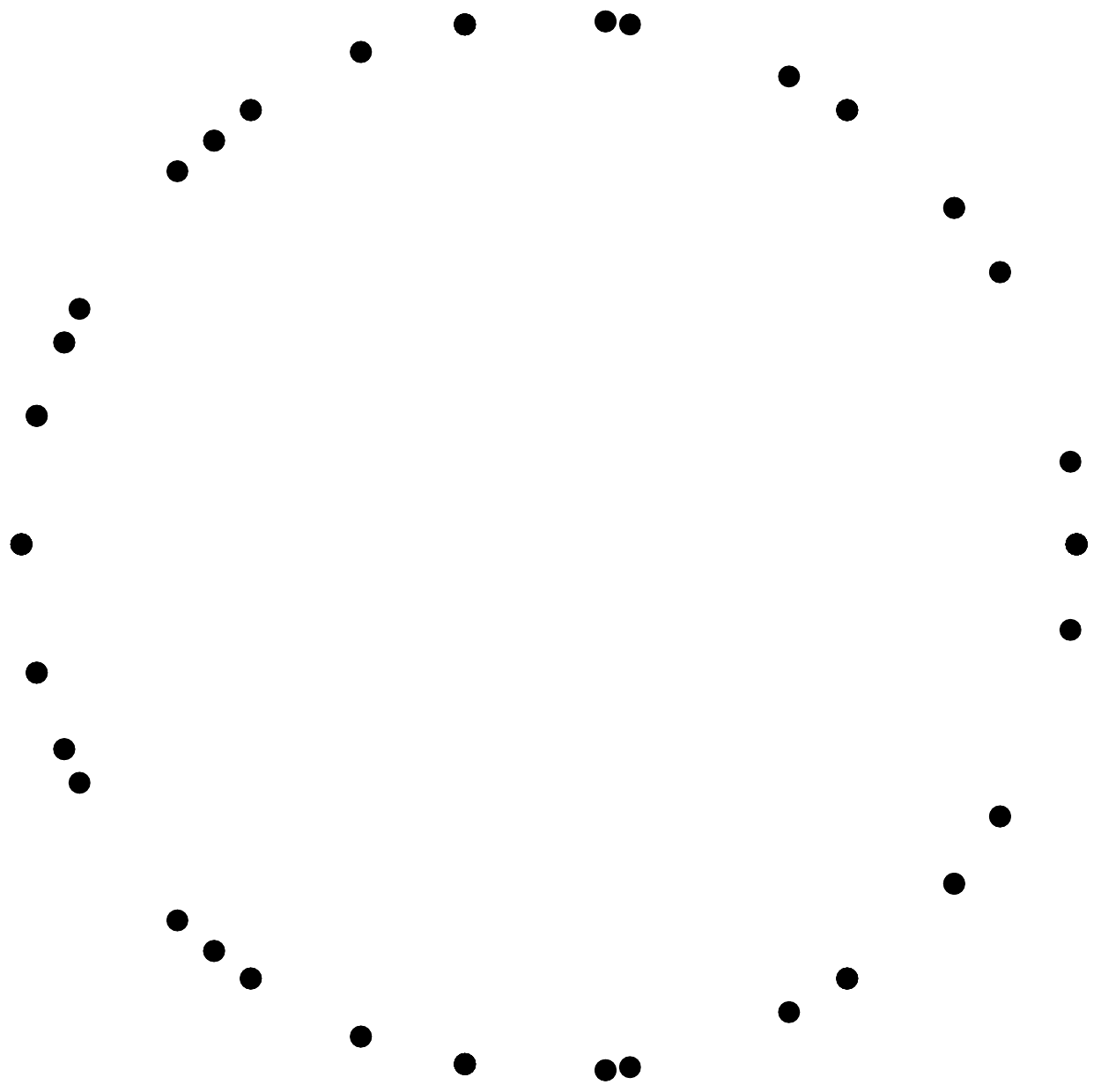}
 \caption{Tangencies $\mathcal A_{T,C_0}$  associated to the Apollonian-9 packing  of Figure~\ref{fig configurations}, with $C_0$ being the bounding circle.}
 \label{fig tangencies}
\end{center}
\end{figure}

We will show that   the cardinality of  $\mathcal{A}_{T,\mathcal{I}}$ grows linearly with $T$: 
\begin{equation}\label{eq:unifdist}
 \#\mathcal{A}_{T,\mathcal{I}}\sim l(\mathcal{I})c_{\pp,C_0}T,\quad T\to \infty,
 \end{equation}
where $l(\mathcal{I})$ is the standard arclength of $\mathcal{I}$, and $c_{\pp,C_0}$ is a constant depending only on $\pp$ and $C_0$. In particular, the tangencies are uniformly distributed in $C_0$. 
To compare, the total number of circles $C\in \mathcal P$ with $\kappa(C)<T$ is $\sim {\rm const}(\mathcal P)T^{\delta}$ for some $\delta>1$ \cite{KO, OS}. 

The   goal of our paper is to study the distribution of nearest neighbour spacings (gaps) in $\mathcal{A}_{\mathcal I}$.  
Let $\{x_{T,\mathcal{I}}^{i}\}$ be the sequence of tangencies in
$\mathcal{A}_{T,\mathcal{I}}$ ordered by some orientation (say, the counter-clockwise orientation). The
nearest-neighbour gaps, or spacings, between the tangencies are
$$d(x_{T,\mathcal{I}}^i,x_{T,\mathcal{I}}^{i+1})$$
where $d$ denotes the arc-length distance.
The mean spacing is
$$ \ave{d_{T}}:= \frac{l(\mathcal{I})}{\#\mathcal{A}_{T,\mathcal{I}}} \sim \frac 1{c_{\pp,C_0}T} \;.
$$
 We define the gap distribution function to be
$$
F_{T,\mathcal{I}}(s)=F_{T;\PP,C_0}(s)=
\frac 1{\#\mathcal{A}_{T,\mathcal{I}}}
\#\{x_{T}^i:\frac{d(x_{T}^i,x_{T}^{i+1})}{ \ave{d_{T}}}\leq s\} \;.
$$

Clearly the definition of $F_{T,\mathcal{I}}$ does not depend on the orientation that we choose.  We will show there is a limiting gap distribution, which is conformally invariant and  independent of $\mathcal{I}$: 
 
\begin{thm}\label{gapdistributionintro}
Given $\mathcal{P},C_0,\mathcal{I}$ as above, there exists a continuous piecewise smooth function $F(s)=F_{\PP,C_0}(s)$, such that
$$\lim_{T\rightarrow\infty}F_{T,\mathcal{I}}(s)=F(s).$$
Moreover, the limit distribution  $F_{\PP,C_0}(s)$ is conformally invariant:  if $M\in SL(2,\CC)$, $M\mathcal{P}=\tilde{\PP}$, $MC_0=\tilde{C_0}$ then
\begin{equation}\label{eq: conformal invariance}
 F_{\tilde{\PP},\tilde{C_0}}(s) = F_{\PP,C_0}(s).
\end{equation}
\end{thm}

Our method does not just give the existence, but actually explicitly calculates the limiting spacing distribution in terms of certain areas, showing that it is a continuous, piecewise smooth function. We also show that $F(s)$ is supported away from the origin, which is a very strong form of ``level repulsion''. 

In \S~\ref{sec:examples} we explicitly compute the limiting distribution for some examples. As a warm-up, we start with   classical Apollonian packings, where we start with three mutually tangent circles and then fill in each curvilinear triangle with the unique circle which is tangent to all three sides of that triangle, and then repeat this process with each newly created curvilinear triangle. In that case we recover a theorem of Hall \cite{Hall} on the gap distribution of Farey points. For a history and further results in this direction  see \cite{CZsurvey}, and see \cite{AC, Marklofsurvey} for   treatments using homogeneous dynamics. 

We then compute the gap  distribution for two other classes of packings: the Guettler and Mallows \cite{GM} packing in Figure~\ref{fig configurations two},  called also Apollonian-3, 
where the curvilinear triangle is filled by three new circles, each tangent to exactly two sides, and for which the associated tangency graph is the octahedron; 
and for the configuration (called Apollonian-9) found by Butler, Graham, Guettler and 
Mallows \cite[Figure 11]{BGGM}, for which the tangency graph is the icosahedron, as in Figure~\ref{fig configurations}. 
 The density functions of these three cases are displayed in Figure~\ref{fig:threeappdensity}.
 \begin{figure}[h]
\begin{center}
\includegraphics[width=12cm]{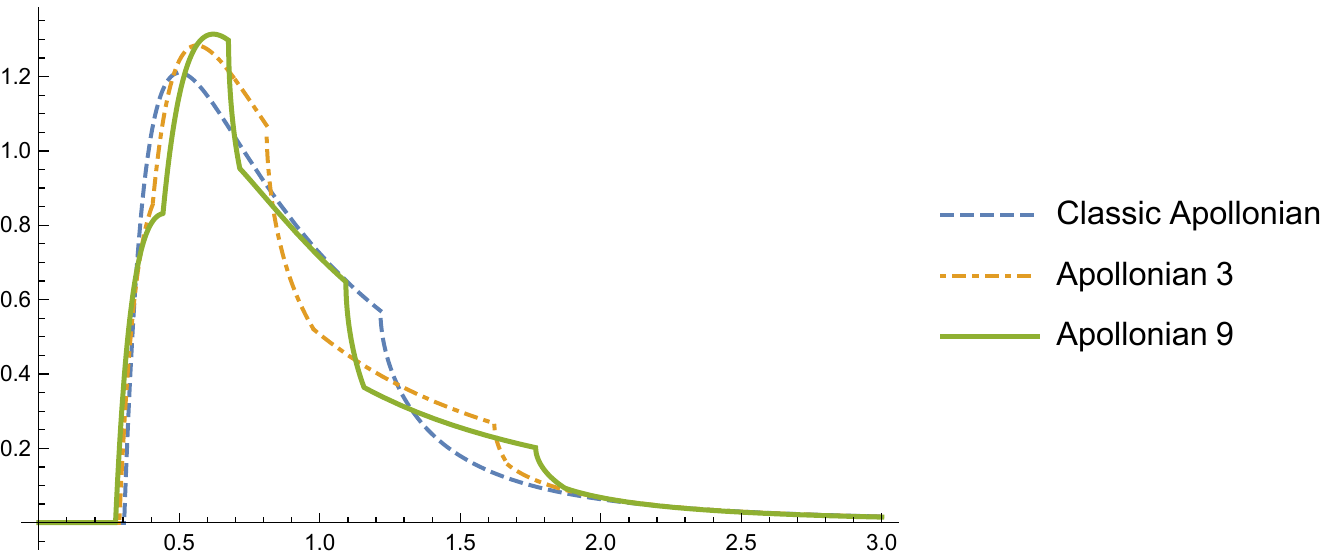}
 \caption{The density $F'(s)$ of the gap distribution for classical Apollonian packings (dashed curve), which is the same as that discovered by Hall for Farey sequence,   for the Apollonian-3 packings of Guettler and Mallows (dotted curve), and for the Apollonian-9 packing of Butler, Graham, Guettler and Mallows (solid curve).}
 \label{fig:threeappdensity}
\end{center}
\end{figure}
 
\subsection{Our method}
We work on models where $\mathcal{K}$ consists of two horizontal lines, with $C_0=\RR$ and $C_1 = \RR+\ii$, and the remaining circles lies in the strip between them.  We call such models generalized Ford configurations.  Let $\Gamma$ be the subgroup of $\mathfrak S$ consisting of orientation preserving elements of $\mathfrak S$  that fix the base circle $C_0$. It is a non-uniform lattice in ${\rm PSL}(2,\RR)$ (this group was used crucially in Sarnak's work \cite{Sarnak} on Apollonian packings). 
Using geometric considerations, we reduce the spacing problem to counting elements  $\begin{pmatrix} a&b\\c&d \end{pmatrix}$ of the 
Fuchsian group $\Gamma$ such that the lower row $(c,d)\in \RR^2$
lies in a dilated region defined by the intersection of  certain quadratic inequalities, and such that $a/c$ lies in a given interval.  
After that, we use the spectral
theory of automorphic forms, specifically the work of Anton Good
\cite{Go83}, to show that the number of lattice points is asymptotically a multiple of the area of this dilated region.

\subsection{Plan of the paper} 
In \S~\ref{sec:Construct} we give  details how to construct $\PP$ from the initial configuration $\KK$.  
We will also construct related groups and prove their geometric properties. 
In \S~\ref{sec: uniform dist of tgts} and \S~\ref{fordgapdistribution} we will prove the uniform distribution of tangencies \eqref{eq:unifdist} 
 and Theorem \ref{gapdistributionintro} in the special case that 
 $\PP$ is a generalized Ford packing, and $C_0$ is one of the two lines from $\PP$.  
In \S~\ref{transfer} we show how to deduce the general case of \eqref{eq:unifdist} and Theorem \ref{gapdistributionintro} from the results of \S~\ref{fordequidistribution} and \S~\ref{fordgapdistribution}. In \S~\ref{sec:examples} we compute some examples. We reduce the computational effort involved by using conformal invariance,   observing  that the packings in the examples admit extra symmetries. For instance, in the case of Apollonian-9 packings, the associated symmetry group is the non-arithmetic Hecke triangle group $\Hecke_5$.

\section{Constructions}\label{sec:Construct}

\subsection{Constructing initial configurations}

We start with a (finite) triangulation  $G$ of the sphere. By the circle packing theorem of Koebe, Andreev and Thurston \cite{StephensonNotices}, there is a circle packing $\mathcal K$ 
of $S^2$ having $G$ as its tangency graph, that is a collection of disks with disjoint interiors on $S^2$ corresponding to the vertices of $G$ (there must be at least $4$ vertices), where two disks are tangent 
if and only if the corresponding vertices are connected in $G$. The interstices, or gaps, between the disks, are the connected components of the complement $S^2\backslash \cup_{C\in \mathcal K}C$ 
of the disks in $S^2$. 
Because we assume that $G$ is a triangulation, that is each face of $G$ is a triangle, this means that the gaps are triangular.

We may identify $S^2$ with the Riemann sphere $\widehat \CC = \CC\cup \infty$, and then by picking a point 
which will correspond to $\infty$ and performing a stereographic projection onto the plane we obtain a circle packing  $\mathcal{K}$ in the finite plane $\CC$. 
If $\infty$ is contained in the interior of one of the disks,  then $\KK$ is realized as a configuration of finitely many circles, with one circle containing the other ones, see Figure~\ref{fig configurations two}. If $\infty$ is a point of tangency of two disks, then $\KK$ is a configuration consisting of two lines and several other circles in between; we call this type of configuration a  generalized Ford configuration, see Figure~\ref{fig:Apollonian-3fordbis}. 

  \begin{figure}[h]
\centering
\begin{minipage}{.5\textwidth}
  \centering
  \includegraphics[width=.7\linewidth]{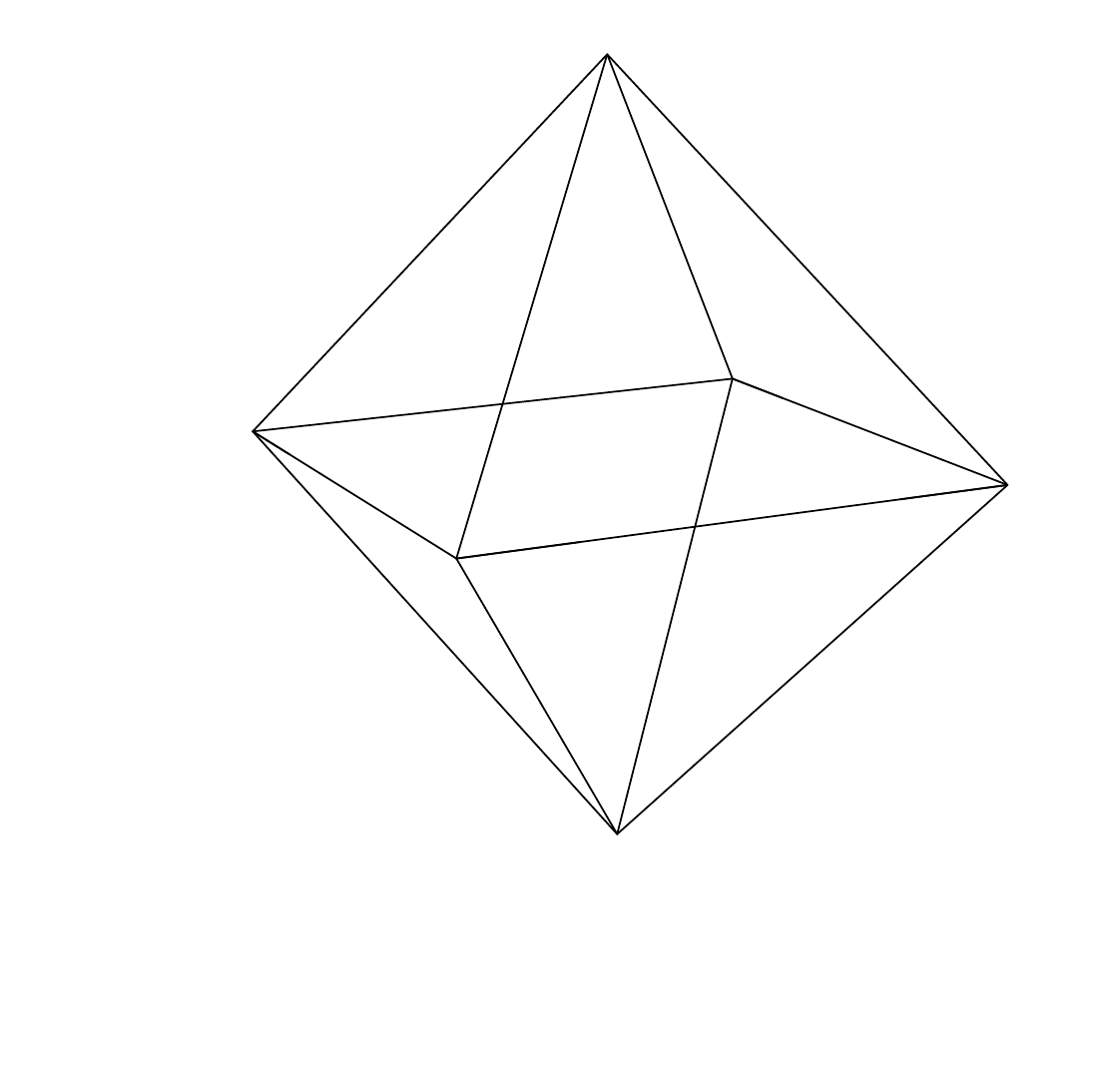}
\end{minipage}%
\begin{minipage}{.5\textwidth}
  \centering
  \includegraphics[width=.7\linewidth]{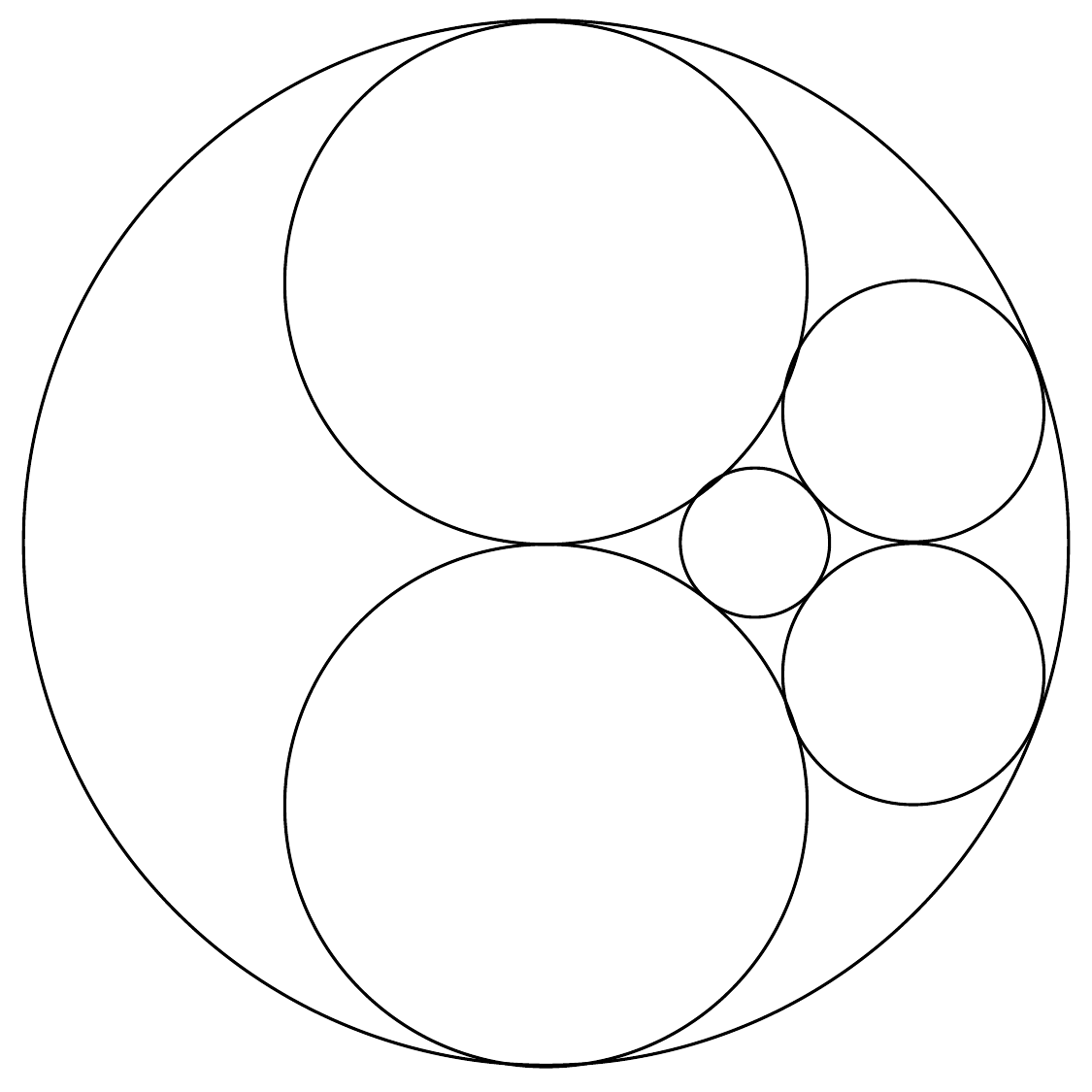}
\end{minipage}
\caption{An initial configuration in the plane, which we call the Apollonian-3 configuration,  for which the associated tangency graph is the octahedron.}
 \label{fig configurations two}
\end{figure}

\subsection{Construction of the packing}
We start with a circle packing $\mathcal{K}$ as above, of $H$ tangent circles
$C_1,C_2,\hdots, C_H$ in the plane, so that the gaps between circles are
curvilinear triangles. We build up a circle packing $\mathcal{P}$
from $\mathcal{K}$ by refections: for each triangular gap, we draw a
dual circle which passes through the vertices of the triangle, see
Figure~\ref{fig:Apo config}.
\begin{figure}[ht]
\begin{center}
\includegraphics[height=5cm, width=5cm]
{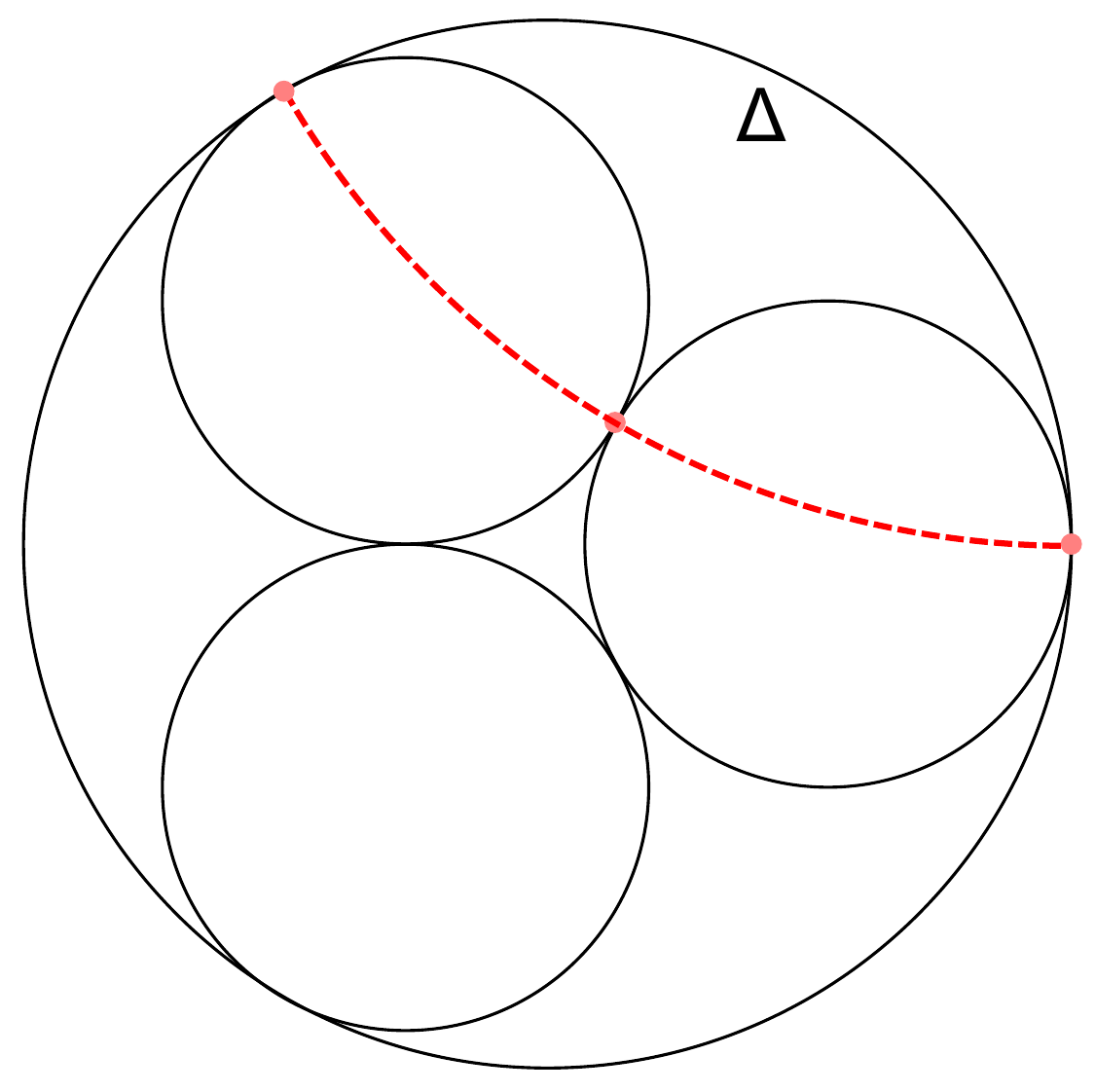}
 \caption{An initial configuration, which generates a classical Apollonian packing.
The dual circle corresponding to the triangular gap labeled $\Delta$ is the dashed circle.}
 \label{fig:Apo config}
\end{center}
\end{figure}
We can generate new circles in the gap by reflecting
$\mathcal{K}$ through the dual circles. Each of the circles forming
the gap is fixed under this reflection, and the remaining circles
are reflected into the gap. Denote these reflections by $\mathcal{S}=\{S_1,S_2,\hdots,S_k\}$ and the corresponding
triangles by $\mathcal{T}=\{\bt_{S_1},\bt_{S_2},\hdots,\bt_{S_k}\}$, where $k$ is
the number of gaps. Note that all these reflections are
anti-holomorphic maps in $\CC$  (throughout this paper, by triangle we mean solid triangle with vertices removed).

Let $\mathfrak{S}=\langle S_1, S_2, \hdots, S_k\rangle$ be the group
generated by the reflections. The reflection group $\mathfrak{S}$
has  only the relations $S_1^2=S_2^2=\hdots=S_k^2=id$. This can be
seen from the Ping Pong Lemma: $\bt_{S_i}$ are disjoint sets, and $S_{j}$ maps every
$\bt_{S_i}$ into $\bt_{S_j}$ when $j\neq i$. 
If we apply $\mathfrak{S}$ to the configuration  $\mathcal{K}$, we obtain a circle packing $\pp$ of infinitely many circles.

Pick a circle $C_0$ from $\PP$, and let $\PP_0$ be the circles
tangent to $C_0$ and let $\mathcal{K}_0=\mathcal{K}\cap
\mathcal{P}_0=\{C_1,\hdots,C_h\}$. We are interested in studying the counting and
spacing problems on $\PP_0$.  The first step is to place $\PP$ in an ambient hyperbolic space with $C_0$ being the boundary, and show that there
is a finite covolume Fuchsian group $\Gamma$ acting on $\PP_0$ and
having finitely many orbits.

\subsection{Construction of $\Gamma$}

A general element in $\mathfrak{S}$ is given by $T_1\cdot \dots T_m$ with $T_i\in\mathcal{S}$.  Therefore each circle $C\in\PP$ can be represented by $ T_1\cdot \dots T_m C_l$, where $C_l\in\mathcal{K}$.  We denote the triangular gap corresponding to $T_i$ by $\bt_{T_i}$.  We define the length of the above word to be $m$.  We say the word above is \emph{minimal} for $C$ if we can not find any shorter expression, and we say $C$ is in $m$-th generation.  We first show that there is a canonical way to express a circle with minimal length: 

\begin{lem}\label{g1}The minimal word for a given circle $C\in\PP$ is unique.  Moreover, $T_1T_2\hdots T_m C_l=C$ is minimal if and only if we have \al{\label{g2}\nonumber&C_l\cap {\bigtriangleup}_{T_m}= \emptyset,\\&T_mC_l\subset{\bigtriangleup}_{T_m},T_{m-1}T_mC_l\subset{\bigtriangleup}_{T_{m-1}},\hdots,T_1\hdots T_mC_l\subset{\bigtriangleup}_{T_1}.}

\begin{proof}
We first prove the ``only if" direction by induction on the generation of $C$. For the base case first generation, if $T_1C_l$ is minimal, then $\bt_{T_1}\cap C_l=\emptyset$, otherwise $T_1$ will fix $C_l$ and then the word is not minimal.  As a result, $T_1C_l\subset\bt_{T_1}$.  Suppose the 
``only if" direction holds for generations up to $m-1$ and $C=T_1T_2\hdots T_m C_l$ is a minimal word for $C$, then $T_2\hdots T_m C_l$ has to be a minimal word for some circle.  By our induction hypothesis, $T_2\hdots T_m C_l\in {\bigtriangleup_{T_2}},\hdots, T_mC_l\subset\bt_{T_m}$ and $C_l\cap\bt_{T_m}=\emptyset$.  Since $T_1\neq T_2$ (otherwise $T_1T_2=id$ and the word is not minimal), we must have $T_1T_2\cdots T_mC_l\subset T_1 {\bigtriangleup}_{T_2} \subset {\bigtriangleup}_{T_1}$. 

From \eqref{g2} we see that $T_1$ corresponds the unique triangle in $\mathcal{T}$ where $T_1C$ sits in, $T_2$ corresponds the unique triangle in $\mathcal{T}$ where $T_2$ sits in ,$\hdots$, $T_k$ corresponds he unique triangle in $\mathcal{T}$ where $T_{k-1}\cdots T_1C$ sits in.  This canonical description shows the uniqueness of minimal expression for a circle. 

We prove the ``if" direction by induction on the length of expression \eqref{g2}.  If $\bt_{T_1}\cap C_l=\emptyset$, then $T_1C_l$ is a circle lying $\bt_{T_1}$. The only circles with shorter words are the ones in $\KK_0$, all of which do not lie in $\bt_{T_1}$, so $T_1C_l$ is minimal.  Now suppose $C=T_1^{'}T_2^{'}\hdots T_{m}^{'}C_{l^{'}}$ is the minimal word for $C$.  Then by our discussion above $T_1^{'}$ is the unique reflection that corresponds to the triangle in $\mathcal{T}$ where $C$ sits in, so $T_1^{'}=T_1$.  Then $T_1C=T_2\hdots T_m C_l$ satisfies the condition \eqref{g2} and the minimal word for $T_1C$ is 
$T_1C=T_2^{'}\hdots T_{m}^{'}C_{l^{'}}$.  By our induction hypothesis, $m^{'}=m, l^{'}=l$, $T_{i}^{'}=T_i$ for each $1\leq i\leq m$. 
\end{proof}

\end{lem}

If $m\geq 1$, we write $\bt_{T_1\cdots T_m}=T_1\cdot \dots \cdot T_{m-1} \bt_{T_m}$, and define $\bt_{C}=\bt{T_1\cdots T_m}$ to be the \emph{minimal} triangle to $C$.  Clearly $C$ is contained in $\bt_{C}$. 

Now let $\Gamma_0$ be the subgroup of $\mathfrak S$ consisting of elements that fix a given circle $C_0$.  By relabeling let $\mathcal{K}_0=\{C_1,C_2,\cdots,C_h\}\subset\mathcal{K}$ be the set of circles that are tangent to $C_0$, and their tangencies are $\alpha_1,\alpha_2,\cdots,\alpha_h$.  Let $\mathcal{S}_0=\{S_{1},\hdots,S_{h}\}\subset\{S_1,\hdots,S_k\}$ be the set of generators of $\mathfrak S$ that fix $C_0$, where $S_{i}\in\mathcal{S}_0$ is the reflection corresponding to the triangle formed by $C_0,C_i,C_{i+1}$.  For our convenience, at this point we make the convention that $C_{i+h}=C_i,\alpha_{i+h}=\alpha_i$ and $S_{i+h}=S_i$ for $i\in\ZZ$, because $C_i$'s from $\mathcal{K}_0$ are forming a loop around $C_0$.

\begin{lem} $\Gamma_0=\langle\mathcal{S}_0\rangle$ is generated by the reflections in $\mathcal{S}_0$.
\begin{proof}
If $\gamma=T_1\hdots T_q$ where $T_1,\hdots,T_q\in \mathcal{S}$ and $\gamma C_0=C_0$, we first absorb the trivial identities $T_{i}T_{i+1}=id$ if $T_i=T_{i+1}$, then we cut the first and last few consecutive letters of the word $T_1\hdots T_q$ that are from $\mathcal{S}_0$ (which may not exist).  We continue this procedure several times until it stabilizes, then we get a word for $C_0$, which has to satisfy condition \eqref{g2} so is minimal.  But the minimal word for $C_0$ is just $C_0$.  Since each cutting corresponds to multiplying an element in $\langle\mathcal{S}_0\rangle$, we have $\gamma\in\langle \mathcal{S}_0\rangle$.
\end{proof}

\end{lem}
\begin{lem}$\Gamma_0$ acts on $\PP_0$, and has finitely many orbits on $\PP_0$.
\end{lem}

\begin{proof} 
$\Gamma_0$ acts on $\PP_0$ because $\mathfrak S$ send pairs of tangent circles to pairs of tangent circles.  Since $\Gamma_0$ fixes $C_0$, it will send any circle tangent to $C_0$ to some circle tangent to $C_0$.  Now suppose $C\in\mathcal{P}_0$ and the minimal word for $C$ is  $C=T_{1}\hdots T_m C_{l}$.  From the argument in Lemma 2.1, $T_{1}$ corresponds to a triangle $\bt_{T_1}\in\mathcal{T}$, which has to be a triangle that contains part of $C_0$ where $C$ lies in, so $T_1$ has to be from $\mathcal{S}_0$.  By the same argument, one can show that $T_{2},\hdots, T_{m}\in \mathcal{S}_0$, and $C_l\in \mathcal{K}_0$.  Therefore, the number of orbits is $h$, the cardinality of $\mathcal{K}_0$.
\end{proof}

\subsection{Geometric Properties of $\Gamma$}

Now we put a hyperbolic structure associated to $\Gamma$.  Without loss of generality we assume that $C_0$ is the bounding circle of radius $1$, whose interior is the unit disk $\DD$, and let $g_i$ be the geodesic connecting $\alpha_i$ and $\alpha_{i+1}$ (again we extend the definition for all $i\in\ZZ$ by setting $g_i$=$g_{i+h}$), so that the region bounded by $g_i$ and $C_0$ contains the triangle $\bt_{S_i}$.  Each $SIi\in\mathcal{S}_0$ preserves the metric in $\DD$,  so $\Gamma_0\subset \text{Isom}(\DD)$, the isometry group of $\DD$.  Let $\mathcal{F}_0$ be the open region bounded by the loops formed $g_i$'s (see Figure~\ref{fig:fundamentaldomain}).  
\begin{figure}[h]
\begin{center}
\includegraphics[width=4cm]{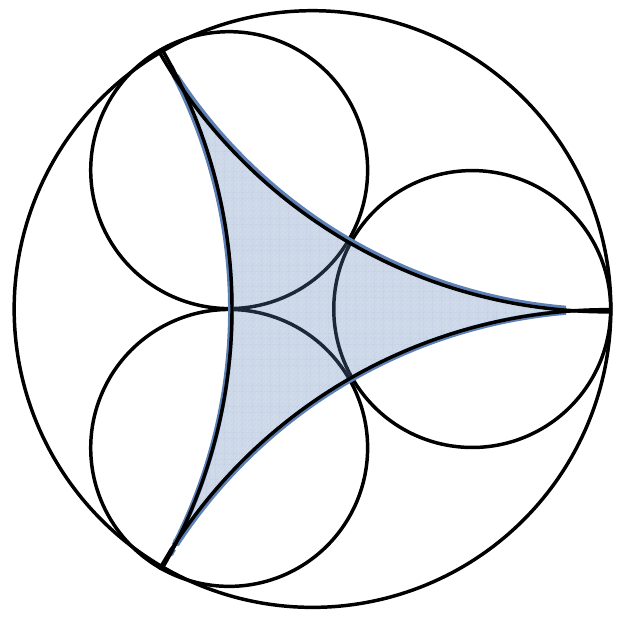}
 \caption{ The fundamental domain $\mathcal{F}_0$ for $\Gamma_0$. }
 \label{fig:fundamentaldomain}
\end{center}
\end{figure}
The following theorem, due to Poincar\'e (see \cite[Theorem 7.1.3]{Ra06}),  
tells us that $\Gamma_0$ acts discontinuously on $\DD$ with fundamental domain $\mathcal{F}_0$.  
\begin{thm}
Let $P$ be a finite-sided, convex polyhedron
in $\DD$    
of finite volume all of whose dihedral angles are submultiples of $\pi$.
Then the group generated by the reflections in the sides of $P$ is a
discrete reflection group with respect to the polyhedron $P$.
\end{thm}

As corollaries, $\Gamma_0$ has $\alpha_1,\hdots,\alpha_{h}$ as cusps,
and the fundamental domain for $\Gamma_0$ has volume $V(\Gamma_0)=\pi(h-2)$ by Gauss-Bonnet.

For our  purpose, we want to take the orientation preserving subgroup of $\Gamma_0$, which we denote by $\Gamma$.  The group $\Gamma$ is a free, index-2 subgroup of $\Gamma_0$ generated by $S_{2}S_{1},\hdots, S_{h}S_{1}$.  To see $\Gamma$ is free, we apply the Pingpong Lemma to $S_{i}S_{1}$, which maps $\bt_{S_j}$ to $\bt_{S_i}$ when $j\neq i$. 
From the properties of $\Gamma_0$, we immediately see: 
\begin{prop}\label{propertyofgamma}
$\Gamma$ has cusps at $\alpha_i$, with stabilizer $\Gamma_{\alpha_i}=\langle S_iS_{i+1}\rangle$. The area of a fundamental domain is  $\area(\Gamma)=2\pi(h-2)$.
\end{prop}


\section{The Uniform Distribution of Tangencies in  Circle Packings}\label{sec: uniform dist of tgts}
We now study the distribution of tangencies on a fixed circle from a circle packing.  Recall that  $\PP$ is generated from a finite circle packing $\KK$, which is associated to a triangulated spherical graph. We pick a circle $C_0\in \mathcal{P}$ and let $\PP_0$ be the collections of circles in $\mathcal{P}$ which are tangent to $C_0$.  Given    a bounded arc  $\mathcal{I}\subseteq C_0$, let 
$\mathcal{A}_{T,\mathcal{I}}$ be the collection of tangencies of circles from $\PP_0$, whose curvatures are bounded by $T$. 
Let $l$ be the standard arclength measure.  We will show   that  
\begin{thm}\label{equidistribution} 
As $T\to \infty$,
\al{\#\mathcal{A}_{T,\mathcal{I}}\sim l(\mathcal{I})c_{\PP,C_0}T.}
with  $c_{\PP,C_0}$ independent of $\mathcal{I}$ given by 
\begin{equation}\label{definitiond}
c_{\PP,C_0}=\frac{D}{2\pi^2(h-2)}, \quad D=\sum_{i=1}^{h}D_i \;,
\end{equation}
where $D_i$ is the hyperbolic area of the region bounded by $C_i$ and the two geodesics linking $\alpha_i$ to $\alpha_{i+1}$ and $\alpha_i$ to $\alpha_{i-1}$ (see Figure~\ref{fig:Dregions}). 

Moreover, $c_{\PP,C_0}$ is conformal invariant:  if $M\in SL(2,\CC)$,
$M(\mathcal{P})=\tilde{\mathcal{P}}$, and $M(C_0)=\tilde{C_0}$, then $c_{\tilde{\PP},\tilde{C_0}}=c_{\PP,C_0}$.  
\end{thm}
We deduce Theorem~\ref{equidistribution}  from some old results of A.~Good \cite{Go83}, 
but it is in principle known; for instance it is a special case of  \cite{OS}. 

We will first prove Theorem~\ref{equidistribution} for a generalized Ford packing, when the initial configuration contains two parallel lines, one of them being the base circle $C_0$. The general case will follow by conformal invariance (Theorem~\ref{transferprinciple}) whose proof is deferred to \S~\ref{transfer}. 
For an explicit   formula for the areas $D_i$ in the case of a Ford configuration, see \eqref{definitiondford}.

\subsection{Ford configurations}
We first show that any of our initial configurations may be brought into a standard shape, that  of a (generalized) Ford configuration, see Figure~\ref{fig:Apollonian-3fordbis}. 
\begin{figure}[h]
\begin{center}
\includegraphics[width=7cm]{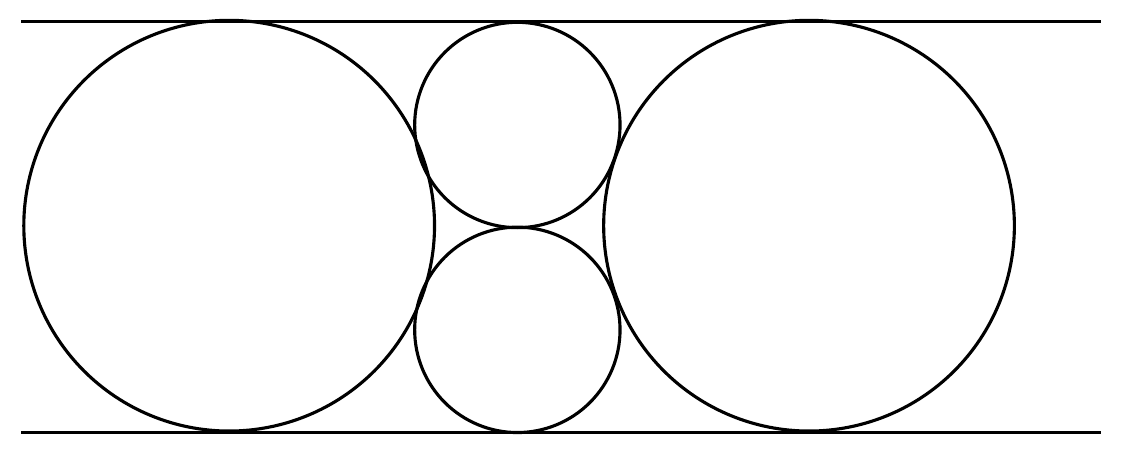}
\caption{The generalized Ford configuration corresponding to the Apollonian-3 configuration of Figure~\ref{fig configurations two}.}
 \label{fig:Apollonian-3fordbis}
\end{center}
\end{figure}
\begin{lem}
We can map the initial configuration $\KK$ by using a M\"obius transformation and possibly a reflection, to one which includes two parallels $C_0=\RR$ and $C_1=\RR+\ii$, two circles tangent to both of $\RR$ and $\RR+\ii$ with their tangencies on $\RR$ is 0 and $\frac{t}{2}$ respectively and some circles in between.
\end{lem}

\begin{proof}
We may use a M\"obius transformation of the Riemann sphere to map the configuration to the interior of the strip $0\leq \Im z\leq 1$ as follows: Pick a triangle in the tangency graph $G$ one of whose vertices corresponds to $C_0$, and whose other vertices corresponds to two other circles $C_1$ and $C_2$, so that $C_0$, $C_1$, $C_2$ form a triple of mutually tangent circles. Setting $\alpha_1=C_0\cap C_1$, $\alpha_2 = C_0\cap C_2$, $\beta  = C_1\cap C_2$ to be the tangency points, corresponding to the edges of the triangle, there is a (unique) M\"obius transformation $M\in SL(2,\CC)$ which maps 
$$ M: \alpha_1\mapsto \infty,\quad \alpha_2\mapsto 0,\quad \beta\mapsto \ii$$ 
  Then $M$ will map $C_0$ and $C_1$ to a pair of tangent lines passing through $0$ and $\ii$ respectively, and $C_2$ will be mapped to a circle tangent to both at these points. Hence $M$ will map   $C_0$ to the real axis $\tilde C_0=\RR$, $C_1$ to the line $\tilde C_1=\RR+\ii$ and $C_2$ to the circle $\tilde C_2=C(\frac{\ii}{2},\frac 12)$ centered at $\ii/2$ and having radius $1/2$, and the remaining circles all lie in the strip $0\leq \Im(z)\leq 1$. 

 The triangular gap $\bt$  between   $C_0$, $C_1$, $C_2$  is mapped into one of the two triangular gaps forming the complement in the strip $0\leq \Im(z)\leq 1$ of the circle $\tilde C_2=C(\frac{\ii}{2},\frac 12)$. After applying if necessary the reflection $z\mapsto -\overline{z}$, we may assume that $\bt$ is mapped to the gap in the half plane $\Re(z)<0$, so that all the other circles in $\mathcal{K}$ are mapped into the gap with $\Re(z)>0$.  Since any triangulation of the sphere has at least $4$ vertices, there must be at least one such circle distinct from $C_0$, $C_1$, $C_2$. Since in the resulting configuration the gaps are still triangular, that means that there is one other circle tangent to both $\RR$ and $\RR+\ii$, say $\tilde C_h=C(\frac{t}{2}+\frac{\ii}{2},\frac 12)$ ($t\geq 2$)  and that the images of all the other circles are mapped to circles lying between $\tilde C_2$ and $\tilde C_h$, as in Figure~\ref{fig:Apollonian-3fordbis}. 
 \end{proof}

\subsection{Uniform distribution of tangencies for generalized Ford packings} \label{fordequidistribution}
We carry over all previous notations. The initial configuration $\KK$ includes two parallel lines $C_0=\RR$ and $C_1=\RR+\ii$, two circles $C_2=C(\frac{\bd{i}}{2},\frac{1}{2})$ and $C_h=C(h+\frac{\bd{i}}{2},\frac{1}{2})$ tangent to both of $\RR$ and $\RR+\ii$, and some circles in between. Our base circle $C_0$ is $\RR$, so $\KK_0=\KK\cap \mathcal{P}_0$ and $\mathcal{P}_0$ is the collection of all circles tangent to $\RR$.  The reflection $S_h$ corresponds to the triangle formed by $\RR$,$\RR+\ii$ and $C_h$.  Applying $S$ to $\KK$, the fundamental domain for $\Gamma$ is formed by geodesics connecting tangencies of $\KK$ and $S\KK$ on $\RR$.  
Each tangency is a parabolic point for $\Gamma$, and since $S_hS_1 0=t$, we have $\Gamma_\infty = \langle S_hS_1 \rangle \pm \begin{pmatrix} 1&t\ZZ\\ & 1\end{pmatrix}$, from Proposition \ref{propertyofgamma}. 

 We record a  computation which will be used at several places: 
\begin{lem} \label{curvatureformula}
Let $M=\mat{a&b\\c&d}\in SL(2,\RR)$. 
i) If $c\neq 0$, then under the M\"obius transform $M$, 
a circle $C(x+y\bd{i}, y)$ will be mapped to the circle 
\al{\label{transformula}C\left(\frac{ax+b}{cx+d}+\frac{r\bd{i}}{(cx+d)^2}, \frac{r}{(cx+d)^2}\right)} 
if $cx+d\neq 0$, 
and to the line $\Im z=1/ 2c^2y$ if $cx+d=0$.
When $c=0$, the image circle is 
\begin{equation*}
C(\frac{ax+b}{d},\frac{r}{d^2}) \;.
\end{equation*}
ii) If $c\neq 0$ then the lines (degenerate circles) $C=\RR+\ii y$ are mapped to $C(\frac ac+\ii \frac 1{2c^2y}, \frac 1{2c^2y})$, and to the line $\RR+a^2 y\ii$ if $c=0$.    
\end{lem}
\begin{proof}
For both i) and ii), if $c\neq 0$, we use the Bruhat decomposition to write $M$ as F
$$M=\mat{1&a/c\\0&1}\mat{c^{-1}&0\\0&c}\mat{0&-1\\1&0}\mat{1&d/c\\0&1}$$
For each factor above the transformation formula is simple, then the curvature of the image circle is obtained by composing these simple formulae together.  The case  when $c=0$ is a simple check.  
\end{proof}

To deduce Theorem \ref{equidistribution}, we need to calculate the contribution from each $C_i\in\KK_0$. For $i\neq 1$, write 
  $C_i=C(\alpha_i+r_i\bd{i},r_i)$. From Lemma \ref{curvatureformula} it follows that $\gamma=\mat{a_{\gamma}&b_{\gamma}\\c_{\gamma}&d_{\gamma}}\in SL(2,\RR)$ sends $C_i$, $i\neq 0,1$, to a circle of curvature $\frac{(c_{\gamma}\alpha_i+d_{\gamma})^2}{r_i}$ and the line $C_1=\RR+\ii $ to a circle of curvature $2c_\gamma^2$.  Therefore, for $i\neq 1$, we need to calculate
\begin{equation}{\label{201502231}
\#\{  \gamma\in\Gamma/\Gamma_{\alpha_i}}: \frac{(c_{\gamma}\alpha_i+d_{\gamma})^2}{r_i}\leq T, \quad 
\frac{a_{\gamma}\alpha_i+b_{\gamma}}{c_{\gamma}\alpha_i+d_{\gamma}}\in\mathcal{I}\}
\end{equation}
and for $i=1$, we need to calculate
\begin{equation}\label{201506281}
\#\{\gamma\in\Gamma/\Gamma_{\infty}: 2c_{\gamma}^2\leq T, \frac{a_{\gamma}}{c_{\gamma}}\in\mathcal{I}\}.
\end{equation}
This counting problem is a special case of a theorem of A.~Good, (see the corollary on Page 119 of \cite{Go83}), which we quote here:
\begin{thm}[Good, 1983] \label{anton}
Let $\Gamma$ be a lattice in $PSL(2,\RR)$. Suppose $\infty$ is a cusp of $\Gamma$ with stabilizer $\Gamma_{\infty}$ and ${\xi}$ is any cusp of $\Gamma$ with stabilizer $\Gamma_{\xi}$.  Choose $M_{\xi}$ so that $M_{\xi}(\xi)=\infty$.  Any $\gamma\not\in \Gamma_{\infty}$ can be written uniquely in the form $\pm\mat{1&\theta_1\\0&1}\mat{0&-\nu^{-\frac{1}{2}}\\\nu^{\frac{1}{2}}&0}\mat{1&\theta_2\\0&1}$ with $\nu>0$, so this determines functions $\theta_1,\theta_2,\nu$ on $PSL(2,\RR)$.  Let $\mathcal{I},\mathcal{J}$ be two bounded intervals in $\RR$. 

Then as $T\to \infty$, 
$$ 
\#\left\{ \gamma^{'}\in  \Gamma M_{\xi}^{-1} :  \nu(\gamma^{'})\leq T,\;  \theta_1(\gamma^{'})\in\mathcal{I},\; \theta_2({\gamma^{'}})\in\mathcal{J}  \right\}
\sim \frac{l(\mathcal{I})l(\mathcal{J})T}{\pi \area(\Gamma)}.
$$
\end{thm}

Now apply Theorem~\ref{anton} to deduce Theorem \ref{equidistribution}.  For $i\neq 1$, we set $\xi=\alpha_i$, and 
$$
M_{\alpha_i}=\mat{a_0&b_0\\-s&s\alpha_i}, \quad s:= \left|\frac{2(\alpha_{i+1}-\alpha_{i-1})}{(\alpha_{i+1}-\alpha_{i})(\alpha_{i-1}-\alpha_{i})}\right|^{1/2},
$$ 
so that  $|M_{\alpha_i}(\alpha_{i+1})-M_{\alpha_i}(\alpha_{i-1})|=\frac{1}{2}$.  Therefore,  
$$
M_{\alpha_i}\Gamma_{\alpha_i} M_{\alpha_i}^{-1}=\left\{\pm\mat{1&n\\0&1}:n\in\ZZ\right\}:=B.
$$  
because $|M_{\alpha_i}(\alpha_{i+1})-M_{\alpha_i}(\alpha_{i-1})|$ contributes half of the fundamental period of the stabilizer at $\infty$ of $M_{\alpha_i}\Gamma M_{\alpha_i}^{-1}$; the other half comes from $|M_{\alpha_i}S_{i+1}(\alpha_{i-1})-M_{\alpha_i}(\alpha_{i+1})|$. 

The map $\gamma\rightarrow \gamma^{'}=\gamma M_{\alpha_i}^{-1}$ is a bijection of the cosets $\Gamma/\Gamma_{\alpha_i}$ and $\Gamma M_{\alpha_i}^{-1}/B$. 
Write $\gamma^{'}=\mat{a_{\gamma^{'}}&b_{\gamma^{'}}\\c_{\gamma^{'}}&d_{\gamma^{'}}}$, then 
$$ 
a_{\gamma^{'}}=(a_{\gamma}\alpha_i+b_{\gamma})s, \quad c_{\gamma^{'}}=(c_{\gamma}\alpha_i+d_{\gamma})s
$$
(when either $\alpha_{i-1}$ or $\alpha_{i+1}$ is $\infty$, this expression is interpreted as a limit).
Then the condition $\frac{a_{\gamma}\alpha_i+b_{\gamma}}{c_{\gamma}\alpha_i+d_{\gamma}}\in\mathcal{I}$ from \eqref{201502231} is the same as $\frac{a_{\gamma^{'}}}{c_{\gamma^{'}}}\in\mathcal{I}$, or $\theta_1(\gamma^{'})\in\mathcal{I}$. The condition $\frac{(c_{\gamma}\alpha_i+d_{\gamma})^2}{r_i}\leq T$ is the same as $c_{\gamma^{'}}^{2}\leq r_is^2T$, or $\nu(\gamma^{'})\leq r_is^2T$.   

Therefore, counting \eqref{201502231} is the same as counting 
$$\#\{\gamma^{'}\in\Gamma M_{\alpha_i}^{-1}:\nu(\gamma^{'})\leq r_i s^2 T,\theta_1(\gamma^{'})\in\mathcal{I},\theta_2(\gamma^{'})\in[0,1)\}$$
Recall the area of $\Gamma$ is $2\pi(h-2)$.  Then by setting $\mathcal{J}=[0,1)$, from Theorem~\ref{anton}  we obtain 
 \aln{\eqref{201502231}\sim\frac
{l(\mathcal{I})r_i s^2}{2\pi^2(h-2)}T\quad \mbox{as } T\to \infty\;.}

We interpret this in conformally invariant terms: 
Let $D_i$ be the area of the region bounded by $C_i$ and the two geodesics $g_{i-1}$, $g_i$ (see Figure~\ref{fig:Dregions}). 
\begin{figure}[h]
\begin{center}
\includegraphics[width=10cm]{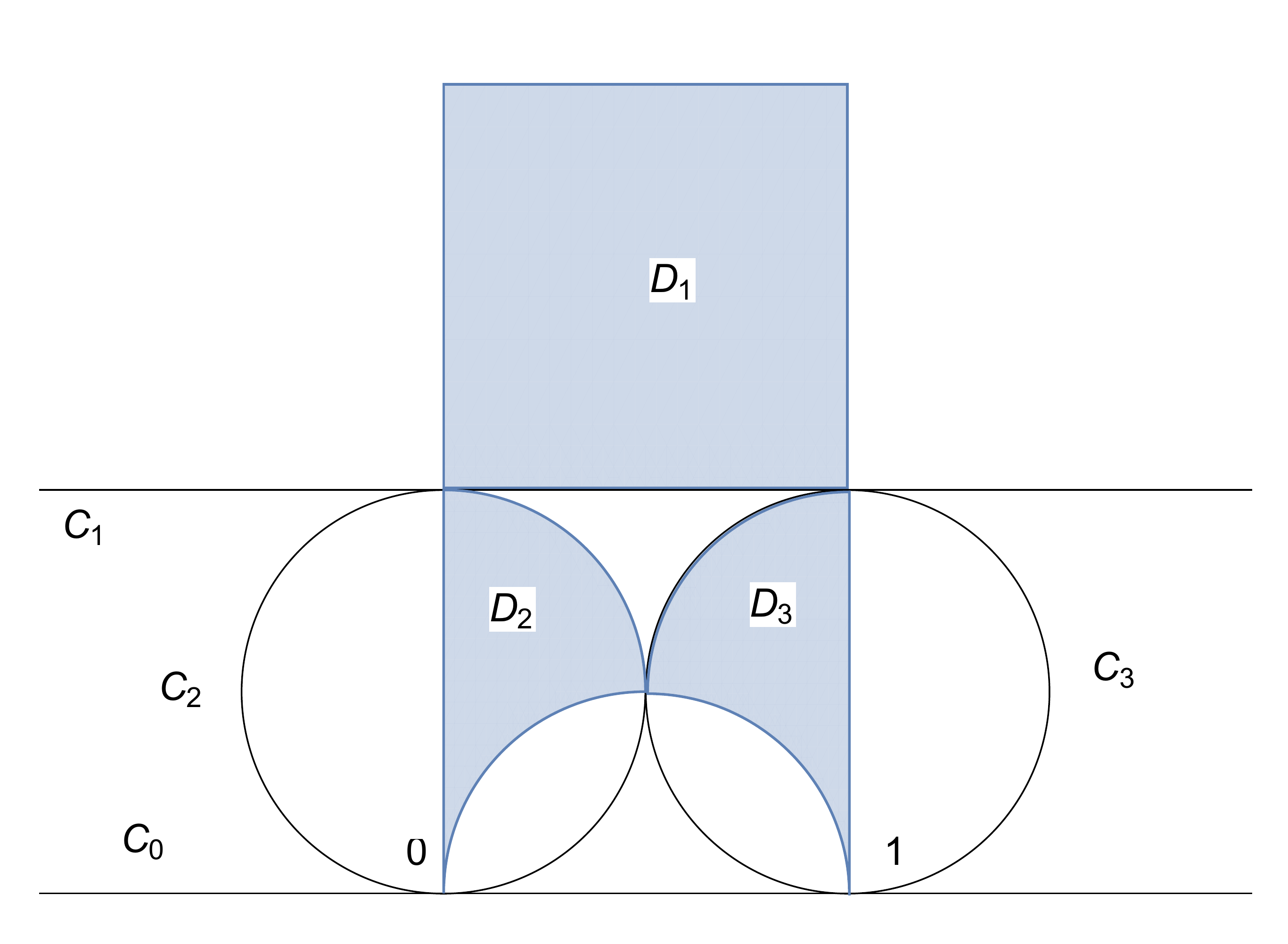}
 \caption{The regions $D_i$  (shaded) bounded by the circles $C_i$ in the initial configurations, and by the two geodesics linking $\alpha_i$ to $\alpha_{i-1}$ and to $\alpha_{i+1}$. Displayed is an initial configuration, generating the classical Ford circles, consisting of two parallel lines $C_0=\RR$ and $C_1=\RR+\ii$, and two mutually tangent circles $C_2=C(\frac{\ii}{2},\frac 12)$ and $C_3=C(1+\frac{\ii}{2},\frac 12)$. Here $\alpha_1=\infty$, $\alpha_2=0$, $\alpha_3=1$. }
 \label{fig:Dregions}
\end{center}
\end{figure}
Then  $M_{\alpha_i}$ will translate this region to a triangle with vertices $\infty$, $M_{\alpha_i}(\alpha_{i+1})$, and $M_{\alpha_i}(\alpha_{i-1})$.  Note that $M_{\alpha_i}(\alpha_{i}+2r\bd{i})$ is a point on the segment connecting $M_{\alpha_i}(\alpha_{i-1})$ and $M_{\alpha_i}(\alpha_{i+1})$, and  that $\Im M_{\alpha_i}(\alpha_i+2r_i\bd{i})=
\frac{1}{2s^2r_i}$.  We also have $|M_{\alpha_i}(\alpha_{i+1})-M_{\alpha_i}(\alpha_{i-1})|=\frac{1}{2}$.  Since $M_{\alpha_i}$ is area preserving, we have 
$$D_i=|M_{\alpha_i}(\alpha_{i+1})-M_{\alpha_i}(\alpha_{i-1})|\cdot 2s^2r_i=s^2r_i=\left|\frac{2r_i(\alpha_{i+1}-\alpha_{i-1})}{(\alpha_{i+1}-\alpha_{i})(\alpha_{i-1}-\alpha_{i})}\right| \;.
$$  
Hence we find 
  \aln{\eqref{201506281}\sim\frac
{l(\mathcal{I})D_i}{2\pi^2(h-2)}T\quad \mbox{as } T\to \infty\;.}

The case of $i=1$ is simpler.  Counting \eqref{201506281} is the same as counting 

\al{\label{horizontalcase}\#\{\gamma\in\Gamma: 2c_{\gamma}^2\leq T, \frac{a_{\gamma}}{c_{\gamma}}\in\mathcal{I}, \frac{d_{\gamma}}{c_\gamma}\in\mathcal{J}\}}
where $\mathcal{J}$ is a fundamental period of $\Gamma_{\infty}$, the length of which is twice the distance of $\alpha_h$ and $\alpha_2$, which is $2\alpha_h$.   
We obtain 
 \aln{\eqref{horizontalcase}\sim\frac
{l(\mathcal{I})2\alpha_h}{2\pi^2(h-2)}\frac{T}{2}\quad \mbox{as } T\to \infty\;.}
Noting that  the area of the region bounded by $C_1$ and the two geodesics $g_0=g_h$, $g_1$ is $D_1=\alpha_h$, we find 
 \aln{\eqref{horizontalcase}\sim\frac
{l(\mathcal{I})D_1}{2\pi^2(h-2)}T\quad \mbox{as } T\to \infty\;.}

Adding up the contributions from
all $C_i\in\KK_0$, we obtain
Theorem~\ref{equidistribution}, with the constant  
\begin{equation*}
c_{\PP,C_0}=\frac{D}{2\pi^2(h-2)}, \quad D=\sum_{i=1}^{h}D_i \;,
\end{equation*}
where $D_i$ is the hyperbolic area of the region bounded by $C_i$ and the two geodesics $g_{i-1}$ and $g_i$.   
 Explicitly, 
\begin{equation}\label{definitiondford}
D_1=\alpha_h,\quad D_i=2r_i\left|\frac{\alpha_{i+1}-\alpha_{i-1}}{(\alpha_{i+1}-\alpha_{i})(\alpha_{i-1}-\alpha_{i})}\right|,\quad i\neq 1
\end{equation} 
(when one of $\alpha_{i-1},\alpha_{i+1}$ is $\infty$, we take the appropriate limiting value of $D_i$).

\section{The Gap Distribution of Tangencies in Circle Packings} \label{fordgapdistribution}

We now study  the gap distribution on $C_0$.  
As in \S~\ref{sec: uniform dist of tgts}, Let $\mathcal{I}\subset C_0$ be an arc (or a bounded interval), and $\mathcal{A}_{T,\mathcal{I}} = \mathcal{A}_T\cap \mathcal{I}$ the tangencies in $\mathcal{I}$ whose circles have curvature at most $T$.  
Let $\{x_{T,\mathcal{I}}^{i}\}$ be the sequence of tangencies in
$\mathcal{A}_{T,\mathcal{I}}$ ordered by counter-clockwise direction. The
nearest-neighbour gaps, or spacings, between the tangencies are
$$d(x_{T,\mi}^i,x_{T,\mi}^{i+1})$$
and the mean spacing is
$$ \ave{d_{T,\mi}}:= \frac{l(\mi)}{\#\mathcal{A}_{T,\mi}}\;.
$$
We define the gap distribution function on
$\mathcal{I}$ to be
$$
F_{T,\mathcal{I}}(s)=\frac 1{\#\mathcal{A}_{T,\mi}}
\#\{x_{T,\mi}^i:\frac{d(x_{T,\mi}^i,x_{T,\mi}^{i+1})}{
\ave{d_{T,\mi}}}\leq s\} \;.
$$

We will find the limiting gap distribution of $\mathcal{A}_{T,\mi}$: 
\begin{thm}\label{gapdistribution}
There exists a continuous piecewise smooth function $F(s)$, which is independent of $\mi$ such that
$$\lim_{T\rightarrow\infty}F_{T,\mi}(s)=F(s).$$
The limiting distribution $F(s)$ is  conformal invariant:  let $M,\tilde{\PP},\tilde{C_0}$ be as in Theorem \ref{equidistribution}, and 
$\tilde{F}$ be the gap distribution function of $\tilde{C_0}$ from
$\tilde{\mathcal{P}}$, then $\tilde{F}=F$. 
\end{thm}
The explicit formula for $F$ is given in Theorem~\ref{thm:formula for F}.

This section is devoted to proving Theorem \ref{gapdistribution} in the case of generalized Ford packings. 

\subsection{Geometric lemmas} 

Recall that we assume the initial configuration $\KK$ includes two parallel lines $\RR$ and $\RR+\ii$, two circles tangent to both of $\RR$ and $\RR+\ii$ with their tangencies on $\RR$ is 0 and $\frac{t}{2}$ respectively and possibly some circles in between. Our base circle $C_0$ is $\RR$, so $\KK_0=\KK\cap \mathcal{P}_0$ and $\mathcal{P}_0$ is the collection of all circles tangent to $\RR$.   

\begin{lem}\label{comparecurv}
i)  Let $K\in\PP_0$ lie in a triangular gap $\bt$ bounded by the real axis $\RR$ and two mutually tangent circles $K_1,K_2\neq \RR$. Then both $K_1$ and $K_2$ have radius greater than that of $K$.  

ii) If $C',C''\in \PP_0$ are such that $C'$ lies in a triangular gap $\bt$ bounded by the real axis, which is disjoint from $C''$, then there is another circle $C\in \PP_0$ which separates  $C'$ and $C''$ and has smaller curvature (i.e. bigger radius) than that of $C'$. 
\end{lem}
\begin{proof}
To see (i), we may move the circle $K$  so that it is tangent to both the real line and the larger of the two initial circles $K_1$ and $K_2$, say it is $K_1$. We get a configuration as in Figure~\ref{fig geom arg}. 
\begin{figure}[h]
\begin{center}
\includegraphics[width=8cm]
{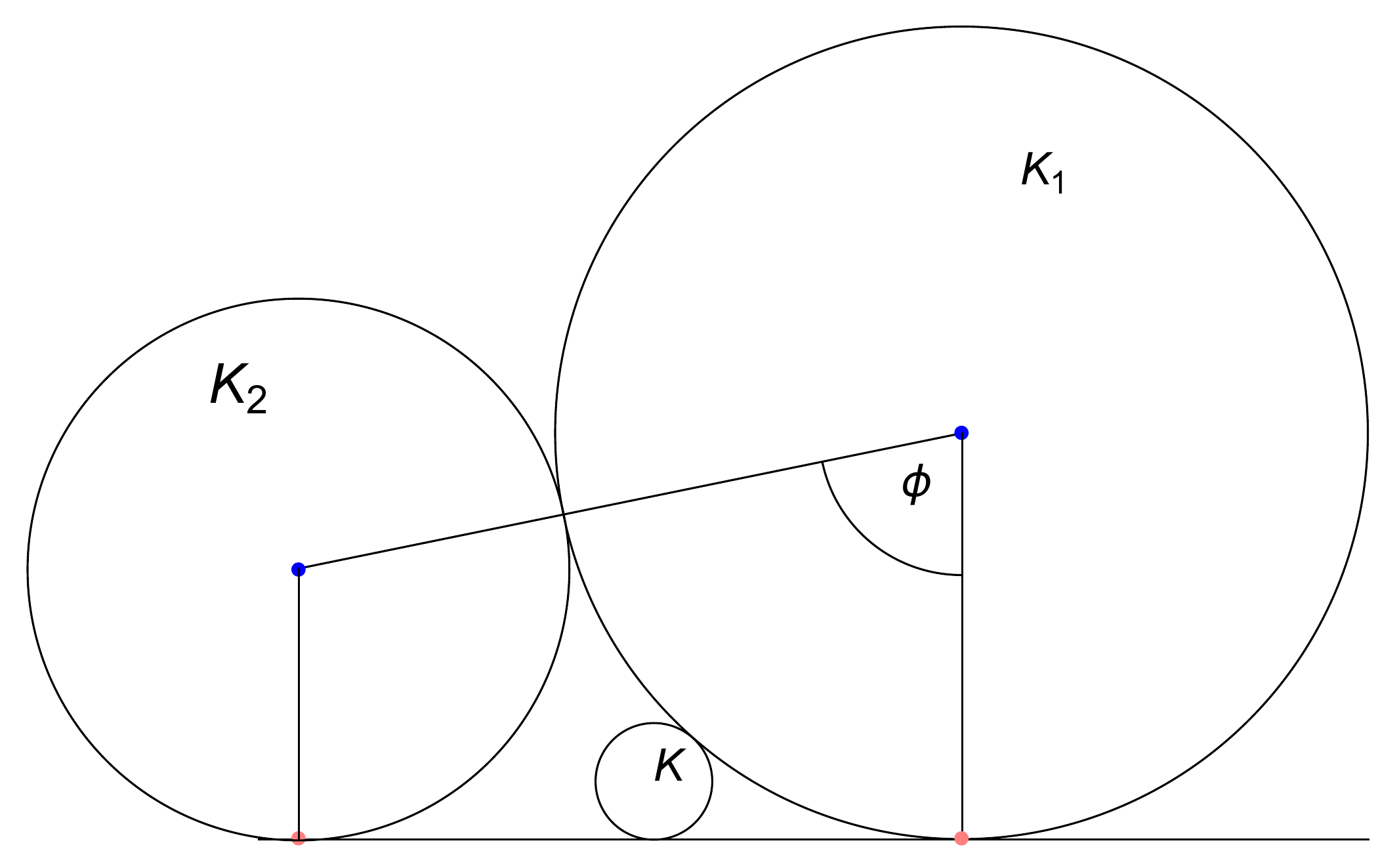}
 \caption{ The radius of any circle contained in a triangular gap formed by two mutually tangent circles, both of which are tangent to the real line, is smaller than the radii of these two circles.}
 \label{fig geom arg}
\end{center}
\end{figure}

If we denote by $r_1$ the radius of the larger circle $K_1$, and by $\phi= \phi(K_1,K_2)$ the angle between the segment joining the centers of the two circles $K_1$, $K_2$, and the segment joining the point of tangency of the circle  $K_1$ with  real line and with the center of $K_1$, then a computation shows that the radius of the third mutually tangent circle $K_2$ is $r(\phi) = (\tan\frac \phi 2)^2r_1$. This is an increasing function of $\phi$. Since the angle $\phi(K_1,K)$ is smaller than $\phi(K_1,K_2)$ for any circle $K$ contained in the triangular gap and tangent to both $K_1$ and the real line, it follows that the circle $K$ has smaller radius than that of $K_2$ as claimed.

\begin{figure}[h]
\begin{center}
\includegraphics[width=8cm]
{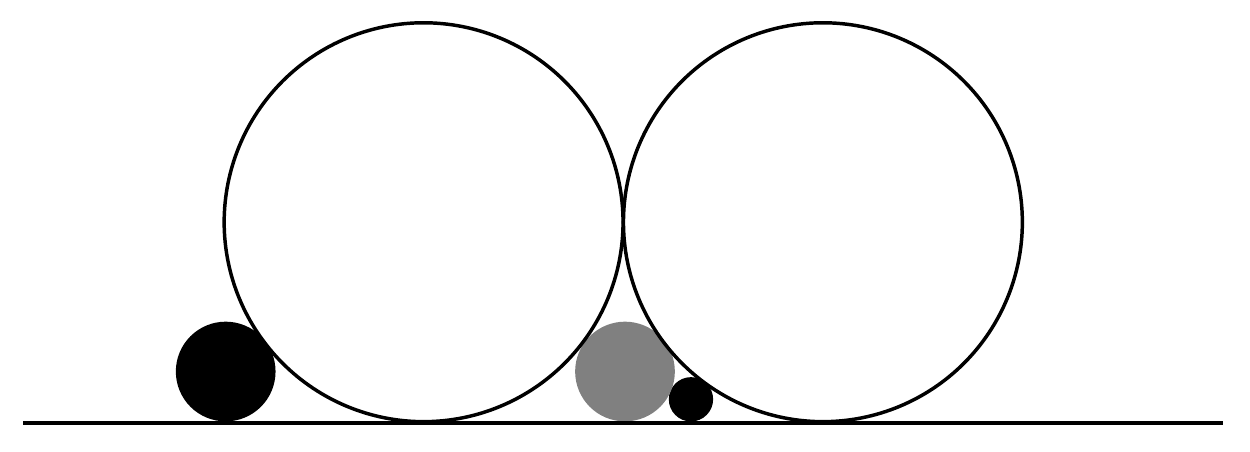}
 \caption{ If $C',C''\in \PP_0$ (shaded black) are such that $C'$ lies in a triangular gap   $\bt$ bounded by the real axis,  with $\bt$ disjoint from $C''$, then there is another circle $C\in \PP_0$ (shaded gray), which is one of the bounding circles of $\bt$, which separates  $C'$ and $C''$ and has smaller curvature  than that of $C'$.}
 \label{geom arg separate}
\end{center}
\end{figure}
 For (ii), note if $C''$ is disjoint from the triangular gap $\bt$ containing $C'$, then  one of the bounding circles $C$ of $\bt$ separates $C'$ and $C''$. By part (i), $C$ has bigger radius than $C'$, see Figure~\ref{geom arg separate}. 
\end{proof}

The following Proposition is a crucial ingredient in understanding the gap distribution:
\begin{prop}\label{12}
For any adjacent pair $\alpha_1^{'},\alpha_2^{'} \in \mathcal{A}_T$ with
corresponding circles $C_1'$ and $C_2'$, there is an element 
$\gamma\in\Gamma_0$ and $C_i,C_j\in \KK_0$ such that
$\gamma(C_i)=C_1'$ and $\gamma(C_j)=C_2'$. 

The element $\gamma\in \Gamma_0$ is unique if $C_1'$, $C_2'$ are disjoint, and if they are tangent then the only other element with this property is $\gamma'=\gamma S_{i,j}$, where  $S_{i,j}\in \mathcal S$ is the reflection in the dual circle corresponding to the triangle formed by $C_i,C_j$ and $\RR$.  
\end{prop}

\begin{proof}
Let $C_{1}^{'}=T_1\hdots T_mC_i$ be the minimal word for $C_1^{'}$, so that the triangle associated to $C_1'$ is $\bt_{C_1^{'}}=T_1\hdots T_{m-1}\bt_{T_m}$.  Let $C_{2}^{'}=T_1^{'}\hdots T_n^{'}C_j$ be the minimal word for $C_2^{'}$ and $\bt_{C_2^{'}}=T_1^{'}\hdots T_{n-1}^{'}\bt_{T_n^{'}}$.  Let's suppose $n\geq m$.  Suppose there exists an integer $k\in[1,m]$ such that $T_{k}\neq T_{k}^{'}$.  We take $k$ to be minimal,  then $T_k\hdots T_mC_i\subset \bt_{T_k}$ and  $T_k^{'}\hdots T_n^{'}C_j\subset \bt_{T_k^{'}}$, and  $\bt_{T_k}$ and  $\bt_{T_k^{'}}$ are two disjoint sets.  Since $T_s=T_s^{'}$ when $1\leq s\leq k-1$, we see that $T_1\hdots T_{k-1}$ maps $\bt_{T_k}$ and $\bt_{T_k^{'}}$ to two disjoint triangles lying in $\bt_{T_1}$.  These two triangles contain $C_1^{'}$ and $C_2^{'}$ respectively.  We have a contradiction here 
because by Lemma~\ref{comparecurv}(ii), there is another circle in $\PP_0$ separating $C_1^{'}$ and $C_2^{'}$ which has smaller curvature,   so $C_{1}^{'}$ and $C_{2}^{'}$ are not neighbors.  Therefore $T_s=T_s^{'}$ when $s\in[1,m]$, and 
$$ C_1'=T_1\dots T_mC_i, \quad C_2'=T_1\dots T_m T'_{m+1}\dots T_n' C_j$$

Now we claim that $T_{s}^{'}$ fixes $C_i$ when $s=m+1,\dots,n$.  This will show that 
$$  C_1'=T_1\dots T_m T_{m+1}'\dots T_n'C_i, \quad C_2'=T_1\dots T_m T'_{m+1}\dots T_n' C_j$$
so that we may take    $\gamma:=T_1^{'}\hdots T_n^{'}$, which satisfies $C_1'=\gamma C_i$ and $C_2'=\gamma C_j$ as required.  

Otherwise, let $k\in [m+1,n]$ again be the smallest integer such that $T_{k}^{'}C_i\neq C_i$, then $\bt_{T_k^{'}}$ contains $T_{k+1}^{'}\hdots T_n^{'}C_j$, and $\bt_{T_k^{'}}$ and $C_i$ are disjoint.  As a result $T_1\hdots T_{k-1}=T_1^{'}\hdots T_{k-1}^{'}$ maps $C_i$ and $\bt_{T_k^{'}}$ to two disjoint sets: $C_{1}^{'}$ and the triangle $\bt=T_1^{'}\hdots T_{k-1}^{'}\bt_{T_k}$, the latter of which contains $C_2^{'}$.  Therefore, by Lemma~\ref{comparecurv}(ii), 
between $C_1^{'}$ and $C_2^{'}$ there must be a circle whose curvature is smaller than that of $C_2^{'}$, so that $C_1'$ and $C_2'$ cannot be adjacent in $\mathcal A_T$, contradiction.  Hence $T_{s}^{'}C_i=C_i$ for $s\in[m+1,n]$ and we have constructed $\gamma:=T_1^{'}\hdots T_n^{'}$. 

To address uniqueness of $\gamma$,  assume $\gamma^{'}$ also sends $C_i$ to $C_1^{'}$, $C_j$ to $C_2^{'}$, then $\gamma^{-1}\gamma^{'}$ fixes $C_i$ and $C_j$ simultaneously.  Therefore, if $C_i$ and $C_j$ are disjoint, $\gamma^{-1}\gamma^{'}=id$; if $C_i$ and $C_j$ are tangent, $\gamma^{-1}\gamma^{'}=id\text{ or }S_{i,j}$, where $S_{i,j}\in \mathcal{S}$ is the reflection corresponding to the triangle formed by $C_i$, $C_j$ and $\RR$. 
\end{proof}

\subsection{Determining which cosets of $\Gamma$ define neighbours in $\mathcal {A}_T$}\label{sec:determining}

We now use Proposition~\ref{12} to relate the gap distribution problem to a question of counting certain cosets in $\Gamma$. 

For each pair of circles $C_i,C_j\in \mathcal K_0$ in the initial configuration which are tangent to $C_0$, we will define regions $\Omega_T^{i,j}$ in the right $cd$-plane $\{(c,d):c\geq 0\}$ as follows.  

a) If $\alpha_i,\alpha_j\neq \infty$, we require 
\begin{equation}\label{bounding}
\frac{(c\alpha_i+d)^2}{r_i}\leq T, \quad \frac{(c\alpha_j+d)^2}{r_j}\leq T
\end{equation} 
 and in addition: 

i) If $C_i$, $C_j$ are tangent, then $(c,d)$ satisfies 
\begin{equation}\label{bounding5}
\mbox{ if }\;(c\alpha_i+d)(c\alpha_j+d)<0 \mbox{ then } 2c^2>T \mbox{ and } 
\frac{(c\alpha_k+d)^2}{r_k}>T,\quad \forall k\neq i,j,\infty
\end{equation}
and 
\begin{equation}\label{bounding2}
\mbox{ if } \; (c\alpha_i+d)(c\alpha_j+d)>0  \mbox{ then }\;\frac{(c\alpha_{i,j}^{(k)}+d)^2}{r_{i,j}^{(k)}}>T\; \quad \forall k\neq i,j\;.
\end{equation} 
where  $S_{i,j}C_k = C(\alpha_{i,j}^{(k)} + \ii r_{i,j}^{(k)}, r_{i,j}^{(k)})$, with  $S_{i,j}$ is the reflection corresponding to the
triangle formed by $C_i,C_j$ and $\RR$.

ii) If $C_i$ and $C_j$ are disjoint (non-tangent), then  $(c,d)$ satisfies 
\begin{multline}\label{bounding6}
\frac{(c\alpha_k+d)^2}{r_k}>T \mbox{ (or }2c^2 >T \mbox{ if }\alpha_k=\infty), \mbox{ if either }  \\
\begin{cases} \alpha_k \mbox{ lies between } \alpha_i, \alpha_j \mbox{ and } (c\alpha_i+d)(c\alpha_j+d)>0\\ \mbox{ or} \\
\alpha_k \mbox{ does not lie between } \alpha_i, \alpha_j \mbox{ and } (c\alpha_i+d)(c\alpha_j+d)<0. 
\end{cases}
\end{multline}

b) If one of $\alpha_i,\alpha_j$ is $\infty$, say $\alpha_j=\infty$, we require 
\begin{equation}\label{boundinginfty}
\frac{(c\alpha_i+d)^2}{r_i}\leq T, \quad 2c^2 \leq T
\end{equation} 
 and in addition: 

i) If $C_i$, $C_j$ are tangent, then $(c,d)$ satisfies 
\begin{equation}\label{bounding5infty}
\mbox{ if }\;\alpha_i=0,d>0 \mbox{ or }\alpha_i=\frac{t}{2}, \frac{ct}{2}+d<0,  \mbox{ then } 
\frac{(c\alpha_k+d)^2}{r_k}>T,\quad \forall k\neq i,j
\end{equation}
and 
\begin{equation}\label{bounding2infty}
\mbox{ if }\;\alpha_i=0,d<0 \mbox{ or }\alpha_i=\frac{t}{2}, \frac{ct}{2}+d>0,  \mbox{ then }\;\frac{(c\alpha_{i,j}^{(k)}+d)^2}{r_{i,j}^{(k)}}>T\; \quad \forall k\neq i,j\;.
\end{equation} 

ii) If $C_i$ and $C_j$ are disjoint (non-tangent), then  $(c,d)$ satisfies 
\begin{equation}\label{bounding6infty}
\frac{(c\alpha_k+d)^2}{r_k}>T \mbox{ if either }   \quad 
\begin{cases} \alpha_k>\alpha_i \mbox{ and } c\alpha_i+d>0\\ \mbox{ or} \\
\alpha_k <\alpha_i \mbox{ and } c\alpha_i+d<0
\end{cases}
\end{equation}

Note that  $\Omega_{T}^{i,j}$ is a finite union of convex sets and $$\Omega_{T}^{i,j}=\sqrt{T}\Omega_{1}^{i,j}\;.$$
 

The result is: 
\begin{prop}\label{omegaij} 
For any $C_i,C_j\in\mathcal{K}_0$ and $\gamma = \begin{pmatrix}a_\gamma&b_\gamma\\ c_\gamma & d_\gamma \end{pmatrix}\in\Gamma$, the circles $\gamma(C_i),\gamma(C_i)$ are neighbors in $\mathcal{A}_{T}$ if and only if $\langle c_{\gamma},d_{\gamma}\rangle$ lies in some $\Omega_{T}^{i,j}$ in the right half plane, defined by  the conditions 

i) If $\alpha_i,\alpha_j\neq \infty$ then  we require \eqref{bounding}, with \eqref{bounding5}  or \eqref{bounding2} in the case that $C_i$, $C_j$ are tangent, and if they are disjoint then we require \eqref{bounding6}. 

ii) If $\alpha_j=\infty$, then  we require \eqref{boundinginfty}, with \eqref{bounding5infty}  or \eqref{bounding2infty} in the case that $C_i$, $C_j$ are tangent, and if they are disjoint then we require \eqref{bounding6infty}.

\end{prop}

\begin{proof}

Let $S\in\mathcal{S}$ be the reflection in the line ${\rm Re}(z)=t/2$, corresponding to the triangular gap bounded by the lines $\RR$, $\RR+\ii$ and the circle with tangency $\frac{t}{2}$. 
From Proposition~\ref{12}, passing to the orientation preserving subgroup $\Gamma$, every (unordered) pair of neighbor circles can be expressed as $\left(\gamma(C_i),\gamma(C_j)\right)$ or $\left(\gamma S(C_i),\gamma
S(C_j)\right)$ for some $\gamma\in\Gamma$ and $C_i,C_j\in\mathcal{K}_0$ with $i\neq j$.  The expression is unique when $C_i$ and $C_j$ are disjoint because the simultaneous stabilizer for $C_i$
and $C_j$ is trivial.  However, we only need to calculate the contribution to $F$ from $(C_i,
C_j)$ under $\Gamma$: the contribution from $(S(C_i),S(C_j))$ is the same as that from $(C_i,C_j)$
because they are mirror symmetric under the reflection of the line {$\Re
z=\frac{t}{2}$.}
Or more formally,  $\Gamma(S(C_i),S(C_j))=S\Gamma S(S(C_i),S(C_j))=S\Gamma(C_i,C_j)$ and $S$ is metric preserving on $\RR$.  

When $C_i$ and $C_j$ are tangent, then
$(S(C_i),S(C_j))$ is also in the orbit of $(C_i,C_j)$ under $\Gamma$.  In this case, every pair of circles from $\Gamma_0((C_i,C_j))$ can be uniquely expressed as $\gamma(C_i,C_j)$ for some $\gamma\in\Gamma$.

If two neighbouring circles in $\mathcal{A}_T$ can be expressed as
$\gamma(C_i),\gamma(C_j)$, then we   need their curvatures to satisfy 
\begin{equation*}
\kappa(\gamma(C_i))\leq T, \quad \kappa(\gamma(C_j))\leq T\;.
\end{equation*} 
By Lemma~\ref{curvatureformula} we have $\kappa(\gamma C_j) = (c\alpha_j+d)^2/r_j$ and so we get \eqref{bounding}.

All the circles in between $\gamma(C_i)$ and $\gamma(C_j)$ must
have curvatures greater than $T$.  We  see from Lemma \ref{comparecurv} that it suffices to check finitely many circles.  
A useful fact here is the following: 
\begin{obs}\label{inbetween}
If neither $C_i$ nor $C_j$ is the horizontal line from $\mathcal{K}_0$, which means $\alpha_i,\alpha_j\neq \infty$, then for another circle $C_l\in \mathcal{K}_0$, $\gamma(C_l)$ is in between $\gamma(C_i)$ and $\gamma(C_j)$ if and only if 
$$ 
 {\rm Sign} (c_{\gamma}\alpha_i+d_{\gamma})(c_{\gamma}\alpha_j+d_{\gamma}) =-{\rm Sign} (\alpha_l-\alpha_i)(\alpha_l-\alpha_j).
 $$

If $C_j$ is the horizontal line, so that $\alpha_j=\infty$, then $\gamma(C_l)$ is in between $\gamma(C_i)$ and $\gamma(C_j)$ if and only if 
$$ 
  {\rm Sign}(c(c\alpha_i+d))= {\rm Sign}(\alpha_l-\alpha_i)
$$
\end{obs} 
\begin{proof}
The observation is that given real numbers $a$ and $b$, a third real $c$ lies between them if and only if 
$${\rm  Sign}(c-a)(c-b) = -1$$
Now take the numbers to be $a=\gamma \alpha_i$, $b\gamma\alpha_i$ and $c=\gamma \alpha_l$ and compute
$$ 
{\rm Sign} (\gamma\alpha_l-\gamma\alpha_i)\cdot(\gamma\alpha_l-\gamma\alpha_j) ={\rm Sign}\frac{(\alpha_l-\alpha_i)(\alpha_l-\alpha_j)}{(c_{\gamma}\alpha_i+d_{\gamma})(c_{\gamma}\alpha_j+d_{\gamma})(c_{\gamma}\alpha_l+d)^2}=-1$$ 
from which the claim follows.

In the case $\alpha_j=\infty$, we have
$$ {\rm Sign} (\gamma \alpha_l-\gamma\alpha_i)(\gamma \alpha_l-\gamma\infty) = -{\rm Sign}\frac{\alpha_l-\alpha_i}{c(c\alpha_i+d)}=-1$$
which gives the claim. 
\end{proof}

There are two cases: \\

Case 1: $C_i$ and $C_j$ are tangent.  
Then $\alpha_k$ does not lie between $\alpha_i$ and $\alpha_j$, for all $k$. 

If $(c\alpha_i+d)(c\alpha_j+d)<0$, then the circles $C_k$ for
$k\neq i,j$ are mapped by $\gamma$ to some circles in between
$\gamma(C_i)$ and $\gamma(C_j)$, and in this case we need
\begin{equation*}
\mbox{ if }\;(c\alpha_i+d)(c\alpha_j+d)<0 \mbox{ then } \kappa(\gamma(C_k))>T,\quad \forall k\neq i,j
\end{equation*}
and using Lemma~\ref{curvatureformula} which gives $\kappa(\gamma(C_k)) = (c\alpha_k+d)^2/r_k$ we obtain condition \eqref{bounding5}.

 If $(c\alpha_i+d)(c\alpha_j+d)>0$, then from Observation \ref{inbetween} none of the circles $\gamma C_k$ lie between $\gamma C_i$ and $\gamma C_j$, but 
the circles $S_{i,j}C_k$ for
$k\neq i,j$ are mapped by $\gamma$ to some circles in between
$\gamma(C_i)$ and $\gamma(C_j)$ (recall that
$S_{i,j}\in\mathcal{S}$ is the reflection corresponding to the
triangle formed by $C_i,C_j$ and $\RR$).  In this situation we need
\begin{equation*}
\mbox{ if } \; (c\alpha_i+d)(c\alpha_j+d)>0  \mbox{ then }\;\kappa(\gamma(S_{i,j}C_k))>T\; \quad \forall k\neq i,j\;.
\end{equation*}  
By  Lemma~\ref{curvatureformula}, $\kappa(\gamma S_{i,j}C_k) = \frac{(c\alpha_{i,j}^{(k)}+d)^2}{r_{i,j}^{(k)}}$  if $S_{i,j}C_k = C(\alpha_{i,j}^{(k)} + \ii r_{i,j}^{(k)}, r_{i,j}^{(k)})$, which gives condition \eqref{bounding2}.

Case 2:  If $C_i$ and $C_j$ are not tangent. We need to make sure
that $\gamma(C_k)$ has curvature $>T$ whenever $\gamma(C_k)$ is in
between $\gamma(C_i)$ and $\gamma(C_j)$.

Again from Observation \ref{inbetween}, if $(c\alpha_i+d)(c\alpha_j+d)>0$, then the circles whose tangencies lie between $\alpha_i$ and $\alpha_j$ will be mapped by $\gamma$ to some circles between $\gamma(C_i)$ and $\gamma(C_j)$; if $(c\alpha_i+d)(c\alpha_j+d)<0$, then the complement of  will be mapped in between $\gamma(C_i)$ and $\gamma(C_j)$.  Thus we need $\kappa(\gamma(C_k))>T$  if either  
\begin{equation*}
  \alpha_k \mbox{ lies between } \alpha_i, \alpha_j \mbox{ and } (c\alpha_i+d)(c\alpha_j+d)>0
  \end{equation*}
  or 
  \begin{equation*}
\alpha_k \mbox{ does not lie between } \alpha_i, \alpha_j \mbox{ and } (c\alpha_i+d)(c\alpha_j+d)<0
\end{equation*}
This gives condition \eqref{bounding6} once we use $\kappa(\gamma(C_k)) = (c\alpha_k+d)^2/r_k$  (Lemma~\ref{curvatureformula}).

From Lemma~\ref{comparecurv}, once \eqref{bounding2}  or \eqref{bounding5} hold when $C_i$, $C_j$ are tangent, or \eqref{bounding6} is satisfied if $C_i$ and $C_j$ are disjoint, then all circles between $\gamma(C_i)$ and $\gamma(C_j)$ have curvatures greater than $T$.  
Putting the above together, we have established Proposition~\ref{omegaij}. 
\end{proof}

\subsection{The contribution of each pair of circles to the gap distribution function} 
We need to calculate the contribution from each pair of circles
$(C_i,C_j)$ to the proportion $F_{T,\mathcal{I}}(s)$ of  gaps of size at most $s$, and we denote this quantity by
$F_{T,\mathcal{I}}^{i,j}(s)$: ,
\begin{equation}\label{def of FijT}
F_{T,\mathcal{I}}^{i,j}(s)=\frac{\#\{(x_{T,\mi}^l,x_{T,\mi}^{l+1})\in
\Gamma(\alpha_i,\alpha_j)
:\frac{d(x_{T,\mi}^l,x_{T,\mi}^{l+1})\}}{\langle d_{T,\mathcal{I}}\rangle}\leq
s\}}{\#\mathcal{A}_{T,\mi}}
\end{equation}
We will later show that these have a limit as $T\to \infty$:
\begin{equation}\label{def of Fij}
F^{i,j}(s):=\lim_{T\to \infty} F_{T,\mathcal{I}}^{i,j}(s)
\end{equation}

A direct computation shows that the distance between $\gamma(C_i)$ and $\gamma(C_j)$ is given by
$\big\rvert\frac{\alpha_i-\alpha_j}{(c\alpha_i+d)(c\alpha_j+d)}\big\rvert$.  The average gap is
asymptotically $1/c_{\PP,C_0}T$.  Therefore, the relative gap condition 
$$
\frac{\big\rvert\frac{\alpha_i-\alpha_j}{(c\alpha_i+d)(c\alpha_j+d)}\big\rvert}{1/c_{\PP,C_0}T(1+o(1))}\leq s
$$ 
in the definition of $F_{T,\mathcal{I}}^{i,j}$ can also be written as
\al{\label{bounding11}\big\rvert(c\alpha_i+d)(c\alpha_j+d)\big\rvert\geq 
\frac{\rvert(\alpha_i-\alpha_j)\rvert c_{\PP,C_0} T}{s}(1+o(1)) \;. }
We will work with a simpler condition
\al{\label{bounding3}\big\rvert(c\alpha_i+d)(c\alpha_j+d)\big\rvert\geq 
\frac{\rvert(\alpha_i-\alpha_j)\rvert c_{\PP,C_0} T}{s} \;. }
which will finally lead to a limiting gap distribution function $F(s)$ which is continuous, then in retrospect working with \eqref{bounding11} will lead to the same limiting function $F(s)$, by the continuity of $F$.

If one of $\alpha_i$ and $\alpha_j$ is $\infty$ (say $\alpha_j=\infty$),  the above is changed to 
\begin{equation}\label{bounding4}
c|(c\alpha_i+d)|\geq \frac{c_{\PP,C_0}T}{s} \;.
\end{equation}


We define a region $\Omega_T^{i,j}(s)$ to  be the elements $(c,d)\in \Omega_T^{i,j}$ satisfying \eqref{bounding3} (or \eqref{bounding4}). Then we have found that 
\begin{equation}
\label{formula for FijT}
 F_{T,\mathcal{I}}^{i,j}(s)=\frac{ 1}{\#\mathcal{A}_{T,\mi}}\#\left\{\gamma = \begin{pmatrix} a_{\gamma}&b_{\gamma}\\c_{\gamma}&d_{\gamma}\end{pmatrix}\in \Gamma: \gamma\alpha_i,\gamma\alpha_j\in\mathcal{I}, (c_{\gamma},d_{\gamma})\in \Omega_T^{i,j}(s)\right \}
\end{equation}

The region $\Omega_T^{i,j}(s)$ is compact because
the condition \eqref{bounding} already gives a compact region
bounded by two sets of parallel lines with different slopes. Note also that
$\Omega_T^{i,j}(s)$ is a finite union of convex sets.

From the defining equation, it is clear that the region $\Omega_T^{i,j}(s)$ grows homogeneously with respect to $T$: 
$$\Omega_T^{i,j}(s)=\sqrt{T}\Omega_1^{i,j}(s)\;.$$  

\subsection{Strong repulsion}
As a consequence of the analysis above, we find  that the normalized gaps in our circle packings all bounded away from zero: 
\begin {cor}
There is some $\delta = \delta(\mathcal P, C_0)>0$ so all gaps satisfy 
$$  {d(x^{i}_{T},x^{i+1}_{T}) c_{\mathcal P, C_0}T } \geq \delta>0$$
\end{cor} 
Thus the limiting distribution $F(s)$ is supported away from the origin.
This is a very strong form of level repulsion, familiar from the theory of the Farey sequence.

\begin{proof}
To prove the assertion, note that we expressed the distribution function $F_T$ as a sum over all (unordered) pairs of distinct circles from the initial configuration $C_i,C_j\in \mathcal K_0$: 
$$ F_T(s) = \  \sum_{C_i,C_j \text{ tangent}} F_T^{i,j}(s)+\sum_{C_i,C_j \text{ disjoint}} F_T^{i,j}(s)$$
where  
$$
F_{T,\mathcal{I}}^{i,j}(s) = \frac 1{\#\mathcal A_{T,\mathcal{I}}} \#\{\begin{pmatrix}a_{\gamma}&b_{\gamma}\\c_{\gamma}&d_{\gamma} \end{pmatrix} \in \Gamma: \frac {(c_{\gamma},d_{\gamma})}{\sqrt{T}}\in \Omega_1^{i,j}(s) \}
$$
Here $\Omega_1^{i,j}(s) $ is the region in the plane of points lying in the compact sets $\Omega_1^{i,j}$ of \S~\ref{sec:determining}, 
satisfying  $|Q_{i,j}(x,y)| >\frac{c_{\mathcal P, C_0}}{s}$. 
$$ \Omega_1^{i,j}(s) = \{(x,y)\in \Omega_1^{i,j}:  |Q_{i,j}(x,y)| >\frac{c_{\mathcal P, C_0}}{s} \}
$$
with 
$$ Q_{i,j}(x,y) = \begin{cases}  (x\alpha_i+y)(x\alpha_j+y),&i,j\neq \infty \\x(x\alpha_j+y),& i=\infty \end{cases}
$$
 The sets $\Omega_1^{i,j}$ are compact, and therefore the functions $|Q_{i,j}(x,y)|$   are bounded on them. Thus if 
$$ \delta(\mathcal P, C_0) =\frac  {c_{\mathcal P, C_0}}{  \max_{i,j}\max (    |( Q_{i,j}(x,y)|: (x,y)\in \Omega_1^{i,j} )}
$$
then $\delta>0$ and $F_T(s)=0$ if $0\leq s\leq \delta$. 
\end{proof} 

In the examples of Section~\ref{sec:examples}, we have $\delta=3/\pi^2=0.303964$ for the classical Apollonian packing \S~\ref{sec:app1} (which reduces to the Farey sequence by conformal invariance), $\delta = 2\sqrt{2}/\pi^2=0.28658$ for the Apollonian-3 packing \S~\ref{sec:app3}, and $\delta =5(1+\sqrt{5})/(6\pi^2)=0.273235$ for the Apollonian-9 packing  \S~\ref{sec:app9}. 

\section{The limiting distribution $F(s)$} \label{sec:limit F}
\subsection{Using Good's theorem}
We now pass to the limit $T\to \infty$, by relating the counting problem encoded in our formula \eqref{formula for FijT} for $F^{i,j}_{T,\mathcal{I}}(s)$ to the area of the region $\Omega_1^{i,j}(s)$. This is done via an equidistribution theorem of A.~Good. To formulate it, recall the Iwasawa decomposition 
$${\rm PSL}(2,\RR)=N^{+}AK$$ 
where
$$
N^{+}=\left\{\mat{1&x\\0&1}\Big\rvert x\in\RR\right\}, \quad
A=\left\{\mat{y^{-\frac{1}{2}}&\\&y^{\frac{1}{2}}}\Big\rvert y>0\right\},
$$
and
$$
K=PSO2=\{\mat{\cos\theta&-\sin\theta\\\sin\theta&\cos\theta}\Big\rvert\theta\in[0,\pi)\}\;.
$$
We can uniquely write any   $\gamma\in PSL(2,\RR)$ as 
\begin{equation}\label{Iwasawa} 
\gamma = \mat{1&x(\gamma)\\0&1} \mat{y(\gamma)^{-\frac{1}{2}}&\\&y(\gamma)^{\frac{1}{2}}}\mat{\cos\theta(\gamma)&-\sin\theta(\gamma)\\\sin\theta(\gamma)&\cos\theta(\gamma)}
\end{equation}
We have the following joint equidistribution according to this decomposition, which is another special case of Good's Theorem 
(see Corollary on Page 119 of \cite{Go83}): 
\begin{thm}[Good\footnote{We state the result slightly differently than in \cite{Go83}: We use $PSL(2,\RR)$ instead Good's original form in $SL(2,\RR)$; 
his parametrization for $K$   differs from ours. 
 }]\label{nak}
Let $\mathcal{I}$ be an bounded interval in $\RR$, $\mathcal{J}$ be an interval in $[0,\pi)$, then as $T\to \infty$, 
$$
\#\Big\{\gamma\in\Gamma: x(\gamma)\in\mathcal{I}\;,\; y(\gamma)<T,\; \theta(\gamma)\in\mathcal{J} \Big\}
\sim \frac{1}{ \area(\Gamma)} l(\mathcal{I}) \frac{l(\mathcal{J})}{\pi}T \;.
$$
\end{thm}

Theorem \ref{nak} allows us to prove the following

\begin{prop}\label{expanding}
For a ``nice" subset $\Omega\subset \{(c,d): c\geq 0\}$,
$$
\#\Big\{\gamma = \begin{pmatrix} *&*\\ c_\gamma &d_\gamma \end{pmatrix}
\in\Gamma_{\infty}\backslash\Gamma: x(\gamma)\in\mathcal{I} \;,\,(c_\gamma,d_\gamma)\in \sqrt{T}\Omega \Big\}
\sim \frac{2l(\mathcal{I})m(\Omega)}{\pi \area(\Gamma)} T
$$
as $T\to \infty$, where $m$ is the standard Lebesgue measure in $\RR^2$. Here ``nice'' means bounded, convex with piecewise smooth boundary. 
\end{prop}
 \begin{proof}
We note that if $\gamma = \begin{pmatrix} *&*\\c&d\end{pmatrix}$ then the Iwasawa decomposition \eqref{Iwasawa} gives $d=y^{1/2}\cos\theta$, $c=y^{1/2}\sin\theta$ so that $(y^{1/2},\theta)$ are polar coordinates in the $(d,c)$ plane. 

We prove Proposition~\ref{expanding} in the special case that
$\Omega$ is bounded by two continuous and piecewise smooth curves
$r_1(\theta),r_2(\theta)$ with $\theta\in[\theta_1,\theta_2]$, and
$r_2(\theta)\geq r_1(\theta)$.  The sign $``="$ is obtained only if when
$\theta=\theta_1\text{ or }\theta_2$.  This special case suffices
for what we need for $\Omega=\Omega_1^{i,j}(s)$.  This is essentially a Riemann sum argument.  First we can express $m(\Omega)$ as
$$
m(\Omega)=\frac 12 \int_{\theta_1}^{\theta_2}r_2^2(\theta)-r_1^2(\theta)d\theta \;.
$$ 
Divide the interval $I=[\theta_1,\theta_2]$ into $n$ equal subintervals 
  $\{I_i|i=1,\hdots,n\}$.  For each subinterval  $I_i$, pick 
$$\theta_{1,i}^{+},\theta_{1,i}^{-},\theta_{2,i}^{+},\theta_{2,i}^{-}$$
at which  $r_1$ and $r_2$ achieve their maximum and minimum respectively.
Let $\Omega_n^{+}$ be the union of truncated sectors
$I_i\times[r_1(\theta_{1,i}^{-}),r_2(\theta_{2,i}^{+})]$, and
$\Omega_n^{-}$ be the union of
$I_i\times[r_1(\theta_{1,i}^{+}),r_2(\theta_{2,i}^{-})]$.  
We have
$$\Omega_n^{-}\subseteq\Omega\subseteq\Omega_n^{+}$$ and
$$
\lim_{n\rightarrow\infty}m(\Omega_n^{-})=\lim_{n\rightarrow\infty}m(\Omega_n^{+})=m(\Omega) \;.
$$

We notice that the statement in Proposition~\ref{expanding} satisfies
finite additivity. From Theorem \ref{nak}, we know Proposition~\ref{expanding} holds for sectors, thus it holds for finite union of
truncated sectors, which is the case for $\Omega_{n}^{+}$ and $\Omega_n^{-}$.
Letting $n$ go to infinity, we prove Proposition~\ref{expanding} for $\Omega$.
\end{proof}

\subsection{The formula for $F(s)$} 

We can now state our main result, the formula for the limiting gap distribution $F(s)$, along the way proving Theorem~\ref{gapdistribution}. We keep our previous notation. 
\begin{thm}\label{thm:formula for F}
For any interval $\mathcal{I}\subset C_0$, $\lim_{T\to \infty} F_{T,\mathcal I}(s) = F(s)$, where 
\begin{equation}\label{whatisf} 
F(s)= \frac{2}{D}\left(\sum_{C_i,C_j \;{\rm tangent}} 
m(\Omega_1^{i,j}(s))+
2\sum_{C_i,C_j \;{\rm disjoint}}m(\Omega_1^{i,j}(s))\right) 
\end{equation}
with $D$ given in \eqref{definitiond}.
\end{thm}

\begin{proof}
We saw that $$F_{T,\mathcal{I}}(s) = \sum_{C_i,C_j\text{ tangent}} F_{T,\mathcal{I}}^{i,j}(s)+ 2\sum_{C_i,C_j\text{ disjoint}} F_{T,\mathcal{I}}^{i,j}(s)$$ and we need to estimate $F_{T,\mathcal{I}}^{i,j}$.  
 From Proposition~\ref{12} and Proposition~\ref{omegaij} we can rewrite 
 $F_{T,\mathcal{I}}^{i,j}$ as 
\al{\frac{\#\{\gamma\in\Gamma: \gamma\alpha_i,\gamma\alpha_j\in\mathcal{I},( c_{\gamma},d_{\gamma})\in\Omega_T^{i,j}\}}{\#\mathcal{A}_{T,\mathcal{I}}}}

In terms of the Iwasawa coordinates \eqref{Iwasawa}, the condition
$\gamma(\alpha_i),\gamma(\alpha_j)\in\mathcal{I}$ is essentially
equivalent to $x(\gamma)\in\mathcal{I}$, when $y(\gamma)$ tends to
$+\infty$, as we shall explain now.  Writing $x = x(\gamma)$, $y=y(\gamma)$ and $\theta = \theta(\gamma)$, we have
$$ \gamma(\alpha) = x+ \frac 1{y}\cdot  \frac{\alpha \cos \theta-\sin \theta}{\alpha \sin \theta + \cos \theta}$$
and since $y = y(\gamma)\gg 1$, this will essentially be $x(\gamma)$ provided $|\alpha \sin \theta + \cos \theta|$ is bounded away from zero.

We first excise small sectors $B_{\beta}$ of angle $\beta$ containing as bisectors those $\theta$'s
such that $\mat{\cos\theta& -\sin\theta\\\sin\theta&
\cos\theta}\alpha_{*}=\infty$, where $*=i$  or $j$.  Then for any $\epsilon>0$, $\exists$ $T_0$ such that if $(c_\gamma,d_\gamma)\in \Omega_{T}^{i,j}$   
and $y(\gamma)>T_0$, then
 $$| \frac 1{y}\cdot  \frac{\alpha_{*} \cos \theta-\sin \theta}{\alpha_{*} \sin \theta + \cos \theta}|<\epsilon$$  
 
 Set 
$$\tilde\Omega:= \Omega_T^{i,j}(s) \backslash (\Omega_{T_0}^{i,j}(s)\cup B_{\beta})$$
so that the area of $\tilde\Omega$ differs from that of $\Omega_T^{i,j}(s) $ by at most $O(\epsilon T)$. 
Then there are two intervals $\mathcal{I}_{\epsilon}^{-},\mathcal{I}_{\epsilon}^{+}$ such that $\mathcal{I}_{\epsilon}^{-}\subset\mathcal{I}\subset\mathcal{I}_{\epsilon}^{+}$, and $$l\left(\mathcal{I}-\mathcal{I}_{\epsilon}^{-}\right), l\left(\mathcal{I}_{\epsilon}^{+}-\mathcal{I}\right)<2\epsilon.$$
and so that  
\begin{multline*}\label{estimateofi}
\#\{\gamma:x(\gamma)\in\mathcal{I}_{\epsilon}^{-},( c_{\gamma},d_{\gamma}) \in\tilde\Omega\}\leq 
\#\{\gamma: \gamma\alpha_{*}\in\mathcal{I},( c_{\gamma},d_{\gamma})\in\tilde\Omega \}\\
\leq \#\{\gamma:x(\gamma)\in\mathcal{I}_{\epsilon}^{+},(c_{\gamma},d_{\gamma})\in \tilde\Omega\} \;.
\end{multline*}

 Applying Proposition~\ref{expanding}, which approximates the above cardinalities by 
$$\frac{2l(\mathcal{I}_\epsilon^{\pm}) m(\tilde \Omega)}{\pi \cdot 2\pi(h-2)} = \frac{2l(\mathcal{I}) m(\Omega_1^{i,j}(s))T}{\pi \cdot 2\pi(h-2)}+O(\epsilon T)$$
(recall $\area(\Gamma) = 2\pi(h-2)$), and dividing  by 
$$\#\mathcal{A}_{T,\mathcal{I}} \sim l(\mathcal{I}) c_{\mathcal{P},C_0}T =  l(\mathcal{I})  \frac{D}{2\pi^2(h-2)}T
$$ 
we obtain 
$$F_{T,\mathcal{I}}^{i,j}(s)=\frac{2m(\Omega_1^{i,j}(s))}{D}(1+o(1))(1+O(\epsilon))+O(\beta)\;.
$$
Letting $\beta,\epsilon\to 0$, we get  
\begin{equation}\label{pf of cont}
\lim_{T\to \infty} F_{\mathcal{I},T}^{i,j}(s) =\frac{2}{D}m(\Omega_1^{i,j}(s)):=F^{i,j}(s)\;,
\end{equation}
with $D$ is as in \eqref{definitiond}, and $m(\Omega_1^{i,j}(s))$ a
piecewise smooth, continuous function of $s$. 
Summing over all pairs of circles in $\mathcal{K}_0$, we get
\begin{equation*}
\begin{split}
  F(s)&=\sum_{C_i,C_j \text{
tangent}}F^{i,j}(s)+2\sum_{C_i,C_j \text{
disjoint}}F^{i,j}(s)\\&=\frac{2}{D}\left(\sum_{C_i,C_j \text{
tangent}}m(\Omega_1^{i,j}(s))+2\sum_{C_i,C_j \text{
disjoint}}m(\Omega_1^{i,j}(s))\right) 
\end{split}
\end{equation*}

The reason there is an extra factor of $2$ for $C_i,C_j$ disjoint is 
because $\Gamma(C_i,C_j)$ only parametrizes half
of the gaps formed by the $\Gamma_0$-orbits of $C_i,C_j$, and the
contribution from $S(C_i)$, $S(C_j)$ is identical  to that  from $C_i$,
$C_j$.
\end{proof}


\section{Conformal invariance}\label{transfer}
Let $M \in SL(2,\CC)$ be a M\"obius transformation. In this section we show that if
a circle packing $\PP$ and a circle $C_0$ from $\PP$ satisfy 
Theorem~\ref{equidistribution} with some constant $c_{\PP,C_0}$, 
and Theorem~\ref{gapdistribution} with some piecewise smooth continuous function
$F$, then the packing $M(\PP)$ and $M(C_0)$ also satisfy 
Theorem~\ref{equidistribution} and Theorem~\ref{gapdistribution} with the
same constant and distribution function.

 In fact we will give a more refined statement: For any pair of distinct circles $C_i,C_j\in \mathcal K_0$  in the initial configuration, 
the densities $F^{i,j}_{\mathcal{I}}(s)$ defined in \eqref{def of FijT},  \eqref{def of Fij} are conformally invariant.
This is particularly useful in reducing the computations of the limiting densities in the examples of \S~\ref{sec:examples} to manageable length. 
The argument itself is routine: The claims are obvious if $M$ is a dilation, and are proved for  general $M\in SL(2,\CC)$ by localizing.  


 \begin{thm}\label{transferprinciple}
Let $\mathcal{I}\subset C_0$ be an arc  (if $C_0$ is a line then take a bounded interval), and $M\in SL(2,\CC)$. 
Assume  $\exists\; 0< b<B$ such that $b<|M^{'}(x)|<B$ for any $x\in\mathcal{I}$. 
 
i) For the packing  $M(\PP)$  and  the base circle $M(C_0)$,
 $$\#\mathcal{A}_{M(\mathcal{I}),T}\sim c_{\PP,C_0}\cdot l(M(\mathcal{I})) \cdot T,\quad {\rm as}\; T\to \infty\;.$$
 
ii) For any pair of distinct circles $C_i,C_j\in \mathcal K_0$  in the initial configuration, the densities $F^{i,j}_{\mathcal{I}}(s)$ 
are conformally invariant: 
\begin{equation}\label{conformal inv of Fij}
F_{T,\mathcal{I}}^{i,j}(s)\sim F_{T,\mathcal{M(I)}}^{i,j}(s),\quad T\to \infty, 
\end{equation}
where on the RHS,  $F_{T,\mathcal{M(I)}}^{i,j}(s)$ refers to the gaps associated with the pair of circles $M(C_i)$, $M(C_j)$ and the packing $M(\PP)$ with base circle $M(C_0)$.

iii) As a consequence, for all $M$ and subarcs $\mathcal{I}\subset C_0$,  
 $$\lim_{T\to \infty}F_{M(\mathcal{I}),T}(s) = F(s).$$
 \end{thm}

 \begin{proof}
We use the notation $``O"$ and $``o"$ throughout the proof related to the asymptotic growth as $T\rightarrow\infty$.
 Without further explanation, all the implied constants depend at most on $M$ and $\mathcal{I}$.

i)  We first show that for a circle $C$ with tangency $\alpha$ on $\mathcal{I}$,
 the curvature of $M(C)$ satisfies 
\al{\label{curvaturetransform} \kappa(M(C))=\frac{\kappa(C)}{|M^{'}(\alpha)|}+O_{M,\mathcal{I}} (1) \;.}
 Choose a parametrization $z(t)=x(t)+\bd{i}y(t)$ for $C$ in a neighborhood of $\alpha$, with $z(0)=\alpha$.  
Then the curvature is given by 
 $$\kappa(C)=\frac{\rvert\Im(\overline{z^{'}(0)}z^{''}(0))\rvert}{|z^{'}(0)|^3}$$
 Therefore, a direct computation shows that 
 $$\kappa(M(C))=\frac{\rvert\Im(\overline{(Mz)^{'}(0)}(Mz)^{''}(0))\rvert}{|(Mz)^{'}(0)|^3}=\frac{\kappa(C)}{|M^{'}(\alpha)|}+O\left(\left\rvert\frac{M^{''}(\alpha)}{M^{'}(\alpha)^2}\right\rvert\right)$$
 which gives  \eqref{curvaturetransform}.\\

 We divide the arc $M(\mathcal{I})$ on $M(C_0)$ into $u$ equal pieces $\mathcal{J}_1,\hdots, \mathcal{J}_u$ with their preimages $\mathcal{I}_1,\hdots,\mathcal{I}_u$ on
 $C_0$. We pick a point $\alpha_i$ from the interval $\mathcal{I}_i$ for each
 $i$, and  let  $\beta_i=M(\alpha_i)$.
 We have
\begin{equation}\label{intervalapproximate}
l(\mathcal{I}_i)=\frac{l(\mathcal{J}_i)}{|M^{'}(\alpha_i)|}\left(1+O\left(\frac{1}{u}\right)\right).
\end{equation}
 For any circle $C$ on $\mathcal{J}_i$, we have
 $$
\kappa(M^{-1}(C))=|M^{'}(\alpha_i)|\kappa(C)\left(1+O\left(\frac{1}{u}\right)\right).
$$

  Therefore, $\exists c_1,c_2>0$ such that
  \al{\label{approximatecircle}
\mathcal{A}_{{\mathcal{I}_i},T|M^{'}(\alpha_i)|\left(1-\frac{c_1}{u}\right)}
  \subset M^{-1}(\mathcal{A}_{\mathcal{J}_i,T})\subset\mathcal{A}_{{\mathcal{I}_i},T|M^{'}(\alpha_i)|\left(1+\frac{c_2}{u}\right)}}
Dividing the above expression by $T$, and letting $T\to \infty$, we
obtain
\begin{multline*}
c_{\PP,C_0}l(\mathcal{I}_i)|M^{'}(\alpha_i)|(1-\frac{c_1}{u})\leq
\liminf_{T\rightarrow\infty}\frac{\#\mathcal{A}_{\mathcal{J}_i,T}}{T} \\
\leq
\limsup_{T\rightarrow\infty}\frac{\#\mathcal{A}_{\mathcal{J}_i,T}}{T}\leq
c_{\PP,C_0}l(\mathcal{I}_i)|M^{'}(\alpha_i)|(1+\frac{c_2}{u})
\end{multline*}

Replace $l(\mathcal{I}_i)$ by \eqref{intervalapproximate}, sum over all $i$,
and let $u\to \infty$, we obtain
$$
\lim_{T\rightarrow\infty}\frac{\#\mathcal{A}_{M(\mathcal{I}),T}}{T}=c_{\PP,C_0}l(M(\mathcal{I})).
$$
Thus we prove the conformal invariance of the density of tangencies.
\\

ii) Now we show that for each pair of circles
$C_p,C_q\in\mathcal{K}_0$,
$$\lim_{T\rightarrow\infty}F_{M(\mathcal{I}),T}^{p,q}(s)=F^{p,q}(s)$$

Let $\{y_{T}^{j,i}\}$ be an ordered sequence of
$\mathcal{A}_{M(\mathcal{I}_i),T}$, and
$\{x_{T|M^{'}(\alpha_i)|}^{k,i}\}$ be the sequence of
$\mathcal{A}_{\mathcal{I},|M^{'}(\alpha_i)|T}$ ordered in a way such that $M(\{x_{T|M^{'}(\alpha_i)|}^{k,i}\})$ has the same orientation as $\{y_{T}^{j,i}\}$.   Let
$\mathcal{A}_{\mathcal{I}_i,|M^{'}(\alpha_i)|T}^{p,q}$ be the subset
of $\mathcal{A}_{\mathcal{I}_i,|M^{'}(\alpha_i)|T}$ consisting of
those $x_{|M^{'}(\alpha_i)|T}^{j,i}$ such that
$(x_{|M^{'}(\alpha_i)|T}^{j,i},
x_{|M^{'}(\alpha_i)|T}^{j+1,i})\in\Gamma(\alpha_p,\alpha_q)$.  Let
$\mathcal{A}_{M(\mathcal{I}_i),T}^{p,q}$ be the subset of
$\mathcal{A}_{M(\mathcal{I}_i),T}$ consisting of those $y_{T}^{j,i}$
such that $(y_{T}^{j,i}, y_{T}^{j+1,i})\in M\Gamma(\alpha_p,\alpha_q)$.

Then
$$F_{M(\mathcal{I}),T}^{p,q}(s)=\frac{\sum_{i=1}^u\#\{y_T^{j,i}\in\mathcal{A}_{M(\mathcal{I}_i),T}^{p,q} :\frac{d(y_T^{j,i},y_T^{j+1,i})}{\frac{l(M(\mathcal{I}))}{\#\mathcal{A}_{M(\mathcal{I}),T}}}\leq s\}}{\sum_{i=1}^u\#\mathcal{A}_{M(\mathcal{I}_i),T}}$$

From \eqref{approximatecircle}, the symmetric difference
$$\#\left(M^{-1}(\mathcal{A}_{M(\mathcal{I}_i),T})\bt\mathcal{A}_{\mathcal{I}_i,|M^{'}(\alpha)|T}\right)=O\left(\frac{T}{u^2}\right)$$
Therefore,  if we let
$$\mathcal{A}_{M(\mathcal{I}_i),T}^0=\{y_{T}^{j,i}:M^{-1}(y_{T}^{j,i})
\text{ and }M^{-1}(y_{T}^{j+1,i}) \text{ are not neighbours in }
\mathcal{A}_{\mathcal{I}_i,|M^{'}(\alpha)|T} \}$$ and
 $$\mathcal{A}_{\mathcal{I}_i,|M^{'}(\alpha_i)|T}^0=\{x_{T}^{k,i}:M(x_{T}^{j,i}) \text{ and }M(x_{|M^{'}(\alpha_i)|T}^{k+1,i}) \text{ are not neighbours in } \mathcal{A}_{\mathcal{I}_i,T} \},$$
then
$$\#\mathcal{A}_{M(\mathcal{I}_i),T}^0=O\left(\frac{T}{u^2}\right)$$
and $$\#
\mathcal{A}_{\mathcal{I}_i,|M^{'}(\alpha_i)|T}^0=O\left(\frac{T}{u^2}\right).$$
Let $\{{y}_T^{j_v,i}\}$ and $\{{x}_{M^{'}(\alpha_i)T}^{k_w,i}\}$ be
the ordered sequence from
$\mathcal{A}_{M(\mathcal{I}_i),T} \backslash \mathcal{A}_{M(\mathcal{I}_i),T}^0$
and
$\mathcal{A}_{\mathcal{I}_i,|M^{'}(\alpha_i)|T} \backslash  \mathcal{A}_{\mathcal{I}_i,|M^{'}(\alpha_i)|T}^0$
respectively.  Then $M$ is a bijection on these two sequences.
Let's say $M(x_{|M^{'}(\alpha_i)|T}^{k_w,i})=y_{T}^{j_v,i}$.
Therefore,
\begin{multline*}
\#\left\{y_T^{j_v,i}\in\mathcal{A}_{M(\mathcal{I}_i),T}^{p,q}\backslash \mathcal{A}_{M(\mathcal{I}_i),T}^0 :\frac{d(y_T^{j_v,i},y_T^{j_v+1,i})}{\frac{l(M(\mathcal{I}_i))}{\#\mathcal{A}_{M(\mathcal{I}_i),T}}}\leq
s\right\} \\=
\#\left\{y_T^{j_v,i}\in\mathcal{A}_{M(\mathcal{I}_i),T}^{p,q}: \frac{d(y_T^{j_v,i},y_T^{j_v+1,i})}{\frac{l(M(\mathcal{I}_i))}{\#\mathcal{A}_{M(\mathcal{I}_i),T}}}\leq
s\right\}+O\left(\frac{T}{u^2}\right)
\end{multline*} 
and 
\begin{multline*}
\#\left\{x_T^{k_w,i}\in\mathcal{A}_{\mathcal{I}_i,|M^{'}(\alpha_i)|T}^{p,q}-\backslash \mathcal{A}_{\mathcal{I}_i,|M^{'}(\alpha_i)|T}^0: \frac{d(x_T^{k_w,i},y_T^{k_w+1,i})}{\frac{l(M(\mathcal{I}_i))}{\#\mathcal{A}_{\mathcal{I}_i,|M^{'}(\alpha_i)|T}}}\leq
s\right\} \\ =
\#\left\{y_T^{j_v,i}\in\mathcal{A}_{M(\mathcal{I}_i),T}^{p,q}: \frac{d(y_T^{j_v,i},y_T^{j_v+1,i})}{\frac{l(M(\mathcal{I}_i))}{\#\mathcal{A}_{M(\mathcal{I}_i),T}}}\leq
s\right\}+O\left(\frac{T}{u^2}\right)
\end{multline*}

Now
\al{\label{415}d(y_T^{j_v,i},y_T^{j_+1,i})=|M^{'}(\alpha_i)|d(x_{|M^{'}(\alpha_i)|T}^{k_w,i},x_{|M^{'}(\alpha_i)|T}^{k_w+1,i})\left(1+O\left(\frac{1}{u}\right)\right),}
\al{\label{416}\frac{l(M(\mathcal{I}_i))}{\#\mathcal{A}_{M(\mathcal{I}_i),T}}=\frac{1}{c_{\PP,C_0}T}\left(1+o(1)\right)}
\al{\label{417}\frac{l({I})}{\#\mathcal{A}_{\mathcal{I},|M^{'}(\alpha_i)|T}}=\frac{1}{c_{\PP,C_0}|M^{'}(\alpha_i)|T}\left(1+o(1)\right)}

Combining \eqref{415}, \eqref{416}, \eqref{417},  we see that
$\exists c_3,c_4$ which only depend on $M$ and $\mathcal{I}$, and
for any arbitrary small number $\epsilon_1,\epsilon_2$, $\exists
T(\epsilon_1,\epsilon_2)$, such that when
$T>T(\epsilon_1,\epsilon_2)$,

\begin{multline*}
\# \left\{x_{|M^{'}(\alpha)|T}^{k_w,i}\in\mathcal{A}_{\mathcal{I}_i,|M^{'}(\alpha_i)|T}^{p,q}:\frac{d(x_{|M^{'}(\alpha)|T}^{k_w,i},x_{|M^{'}(\alpha)|T}^{k_w+1,i})}{\frac{l(\mathcal{I})}{\#\mathcal{A}_{\mathcal{I},|M^{'}(\alpha_i)|T}}}\leq s(1-\epsilon_1)\left(1-\frac{c_3}{u}\right)\right\}\nonumber\\\nonumber
\leq\# \left\{y_T^{j_v,i}\in\mathcal{A}_{M(\mathcal{I}_i),T}^{p,q}:\frac{d(y_T^{j_v,i},y_T^{j_v+1,i})}{\frac{l(M(\mathcal{I}))}{\#\mathcal{A}_{M(\mathcal{I}),T}}}\leq s\right\} 
\\
\leq\# \left\{x_{|M^{'}(\alpha)|T}^{k_w,i}\in\mathcal{A}_{\mathcal{I}_i,|M^{'}(\alpha_i)|T}^{p,q}:\frac{d(x_{|M^{'}(\alpha)|T}^{k_w,i},x_{|M^{'}(\alpha)|T}^{k_w+1,i})}{\frac{l(\mathcal{I})}{\#\mathcal{A}_{\mathcal{I},|M^{'}(\alpha_i)|T}}}\leq
s(1+\epsilon_2)\left(1+\frac{c_4}{u}\right)\right\}
\end{multline*}

We also have
\begin{equation*}
\#\mathcal{A}_{M(\mathcal{I}_i),T}=\#\mathcal{A}_{\mathcal{I}_i,|M^{'}(\alpha_i)|T}\left(1+O\left(\frac{1}{u}\right)\right)
\end{equation*}
Therefore, $\exists c_5,c_6>0$, 
\begin{multline*}
 F_{\mathcal{I},|\mathcal{M}^{'}(\alpha_i)|T}^{p,q}\left(s(1-\epsilon_1)\left(1-\frac{c_3}{u}\right)\right)-\frac{c_5}{u}\\
\leq
\frac{\#\{y_T^{j,i}\in\mathcal{A}_{M(\mathcal{I}_i),T}^{p,q}: \frac{d(y_T^{j,i},y_T^{j+1,i})}{\frac{l(M(\mathcal{I}_i))}{\#\mathcal{A}_{M(\mathcal{I}_i),T}}}\leq
s\}}{\#\mathcal{A}_{M(\mathcal{I}_i),T}}\\
\leq
F_{\mathcal{I},|\mathcal{M}^{'}(\alpha_i)|T}^{p,q}\left(s(1+\epsilon_2)\left(1+\frac{c_4}{u}\right)\right)+\frac{c_6}{u}
\end{multline*}
As a result, as $T\to \infty$, we obtain
\begin{multline*} 
F^{p,q}\left(s(1-\epsilon_1)\left(1-\frac{c_3}{u}\right)\right)-\frac{c_5}{u}\leq
\liminf_{T\rightarrow\infty} F_{M(\mathcal{I}_i),T}^{p,q}(s) \\
\leq
\limsup_{T\rightarrow\infty} F_{M(\mathcal{I}_i),T}^{p,q}(s) \leq
F^{p,q}\left(s(1+\epsilon_2)\left(1+\frac{c_4}{u}\right)\right)+\frac{c_6}{u}
\end{multline*}

Since $F_{M(\mathcal{I}),T}^{p,q}(s)$ is a convex combination of
$F_{M(\mathcal{I}_i)}^{p,q}(s)$, if we let $u\rightarrow\infty$,
$\epsilon_1,\epsilon_2\rightarrow0$ and use the continuity of $F^{p,q}$ (as follows from \eqref{pf of cont}) ,
we obtain
$$\lim_{T\rightarrow \infty}F_{M(\mathcal{I}),T}^{p,q}(s)=F^{p,q}(s).$$
\end{proof}

\subsection{Proof of Theorems~\ref{equidistribution} and \ref{gapdistribution}}
Any bounded circle packing can be written as $M(\PP)$ for some
generalized Ford packing $\PP$. The only issue we need to deal with
is $\infty$, if an arc $\mathcal{J}$ in $M(C_0)$ contain $M\infty$,
the argument in the proof of Theorem \ref{transferprinciple} does
not directly apply.  However, this is easily solved by precomposing $M$
with some $X\in\Gamma$, so
that  $(MX)^{-1}(\mathcal{J})$ is a bounded arc, and $b<|(MX)^{'}(\alpha)|<B$ for some $b,B>0$ and any $\alpha\in (MX)^{-1}(\mathcal{J})$.  Then we can apply Theorem
\ref{transferprinciple}.

\section{Examples}\label{sec:examples}

We compute the gap distribution function $F(s)$ for three examples: The classical Apollonian packing, where we start with a configuration of four mutually tangent circles, whose tangency graph is the tetrahedron; the Apollonian-3 packing  in which one starts with three mutually tangent circles, and in either of the curvilinear triangles formed by three mutually tangent
circles we pack three more circles, forming a sextuple whose tangency graph is the octahedron; and the Apollonian-9 packing  introduced in \cite{BGGM} where the initial configuration consists of 12 circles, with the icosahedron as its tangency graph. 
 
It would be interesting to apply the method to compute the gap distribution for other examples, such as the ball-bearing configurations of \cite{GM}.

\subsection{The computational procedure} 
We   explain the procedure used in the computations:   
According to Theorem~\ref{thm:formula for F}, the gap distribution $F(s)$ is given in terms of a sum of functions $F^{i,j}(s)$ over (unordered) pairs of distinct circles $C_i, C_j\in \mathcal K_0$ of areas of regions  $\Omega^{i,j}(s)=\Omega^{i,j}_1(s)$ in the $(c,d)$ plane, explicitly described in ~\S~\ref{sec:determining}.  
So for instance in the case of the Apollonian-9 packing, there are $\binom{5}{2} = 10$ such pairs. We are able to cut down on the computational effort involved by using conformal invariance, once we note that the packings we study below enjoy more symmetries than a generic packing.
 According to Theorem~\ref{transferprinciple}, if there is some $M\in SL(2,\CC)$ which preserves $\PP_0$ and $C_0$ (note that $F(S)$ only depends on $\PP_0$, not on all of $\PP$), and takes the pair of circles $(C_i,C_i)$ in $\mathcal{K}_0$ to another pair $(MC_i,MC_j) =(C_{k},C_{\ell})$ in $\mathcal{K}_0$, then $F^{i,j}(s) = F^{k,\ell}(s)$. 
 
 In each of these packings, there exists a group $\tilde{\Gamma}$ which is larger than $\Gamma$ that also acts on this given packing.     In each of the three examples presented, the bigger group $\tilde \Gamma$ is one of the Hecke triangle groups $\Hecke_q$, which is the group of
M\"obius transformations generated by the inversion $S:z\mapsto
-1/z$ and the translation $T_q:z\mapsto  z+\lambda_q$, where
$\lambda_q = 2\cos \frac \pi q$:
$$ 
\mathbb{G}_q=\left\langle T_q=\mat{1&2\cos\frac{\pi}{q}\\0&1},\quad  S=\mat{0&-1\\1&0}\right\rangle
$$
For $q=3$ we recover the modular group: $\Hecke_3 = PSL(2,\ZZ)$. 
For $q\neq 3,4,6$ the groups $\Hecke_q$ are non-arithmetic.

In all three cases, we show that $\Gamma \trianglelefteq  \Hecke_q$ is a normal subgroup of the appropriate $\Hecke_q$. 
In the first example of the classical Ford packing, 
$\Gamma\triangleleft\Hecke_3=PSL(2,\ZZ)$ which acts transitively on the Ford circles $\PP_0$.  
In the second example, of the Apollonian-3 packing, $\Gamma\triangleleft \Hecke_4$. 
In the third example, the Apollonian-9 packing, $\Gamma\triangleleft \Hecke_5$.

  It is a simple check that in each of these cases, the generators $\mathbb{G}_q$ send circles in $\KK_0$ to some circles in $\PP_0$.  So if we show $\mathbb{G}_q$ normalizes $\Gamma$, then it follows that $\mathbb{G}_q$ permutes the circles in  $\PP_0=\Gamma\mathcal{K}_0$.   
   We further show that $\Hecke_q$ acts transitively on pairs of tangent circles in $\KK_0$  for  $q=3,4,5$, and on pairs of disjoint circles  for $q=4,5$.   Hence by conformal invariance of the components $F^{i,j}(s)$, this gives an expression for $F(s)$ as a sum of areas of two regions for $q=4,5$  (only one for $q=3$). 
  
  These regions are in turn expressed as a finite union of certain subregions $Z_k(s)$. For each subregion $Z_k\subseteq \Omega^{i,j}$, the first two equations give the conditions $\kappa(C_i)\leq1,\kappa(C_j)\leq1$, and the other equations except the last one gives the condition $\kappa(C)>1$, where $C$ exhausts circles in $\mathcal{K}_0$ such that $\gamma(C)$ is in between $\gamma(C_i)$ and $\gamma(C_j)$. The last quadratic equation corresponds to conditions  \eqref{bounding3} or \eqref{bounding4}, capturing relative gap information.

  The areas of these regions can be explicitly computed in elementary terms, but the resulting formulae are too long to record. Instead we display plots of  the density of the normalized gaps (the derivative of $F$).
\\

\subsection{Classical Ford Circles.}\label{sec:app1}
We start with a configuration of four mutually tangent circles in Figure~\ref{apford}: The base circle is $C_0=\RR$, and the other circles in the initial configuration are
\begin{figure}[h]
\begin{center}
\includegraphics[width=8cm]
{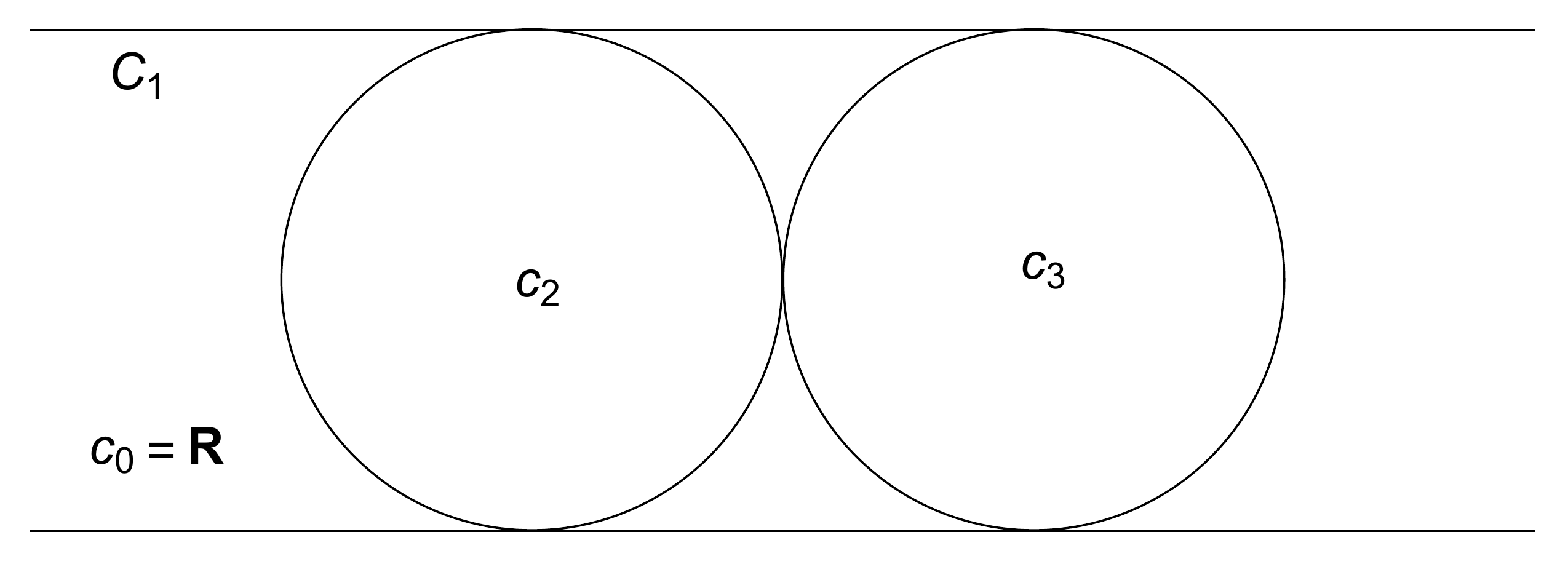}
 \caption{Ford Configuration of Apollonian type}
 \label{apford}
\end{center}
\end{figure}
$$C_1=\RR+\bd{i}, \quad C_2=C(\frac{\bd{i}}{2},\frac{1}{2}), 
\quad C_3=C(1+\frac{\ii}{2},\frac{1}{2})\;.
$$ 

The resulting packing $\PP$ is a classical Apollonian packing, generated by inscribing a unique circle   into every curvilinear triangle
produced by the four mutually tangent circles and repeating. The circles $\PP_0$ tangent to the real line are precisely the classical Ford circles $C(\frac pq, \frac 1{2q^2})$, 
having as points of tangency the Farey sequence $\{p/q: q>0, \gcd(p,q)=1\}$. The gap distribution of the Farey sequence was found by Hall \cite{Hall}, 
and we carry out the computation here as a warm-up to illustrate our method. 

The value of $D$ is $3$ according to \eqref{definitiondford}, see Figure~\ref{fig:Dregions}, and thus $c_{\PP,C_0} = D/2\pi^2(h-2) = 3/2\pi^2$.  

 The group $\Gamma_0$ is generated by the three reflections 
$$z\rightarrow  \mat{-1&0\\0&1}\bar{z}, \quad  z\rightarrow  \mat{-1&2\\0&1}\bar{z},\quad z\rightarrow  \mat{1&0\\2&-1}\bar{z}
$$ 
so $\Gamma$ generated by $\mat{1&2\\0&1}$ and $\mat{1&0\\2&1}$, which are the generators of the principal congruence subgroup 
$$
\Gamma(2)=\{\gamma=\mat{a&b\\c&d}\in PSL(2,\ZZ):\gamma=I\bmod 2\}
$$ 
which is a normal subgroup of $PSL(2,\ZZ) = \Hecke_3$.    
Any pair of circles from $\{C_1,C_2,C_3\}$ are {\it conformally equivalent}, in the sense that there exists some $\gamma\in \Hecke_3=PSL(2,\ZZ)$ which preserves the packing and maps $C_1$, $C_2$ to this given pair of circles: Indeed,  
$$T_3(C_1,C_2)=(C_1,C_3), \quad T_3S(C_1,C_3)=(C_3,C_2)$$
Therefore we have  
 $F^{1,3}(s)=F^{2,3}(s)=F^{1,2}(s)$. 
and so by Theorem~\ref{thm:formula for F}, 
$$ F(s) = 3F^{1,2}(s) = \frac 23 \cdot 2 m(\Omega^{1,2}(s))$$
and it remains to compute the area of $\Omega^{1,2}(s)$.

We compute the region $\Omega_1^{1,2}(s)$ using \eqref{boundinginfty}, \eqref{bounding5infty}, \eqref{bounding2infty}.  
A direct computation gives 
$$\kappa(\gamma(C_1)=2c^2, \quad \kappa(\gamma(C_2))=2d^2,\quad  \kappa(\gamma(C_3))=2(c+d)^2\;.
$$
The distance between $\gamma\alpha_1=\gamma\infty$ and $\gamma\alpha_2=\gamma0$ is $1/|cd|$.  When $d>0$, $C_3$ will be mapped by $\gamma$ to some circle in between $\gamma(C_1)$ and $\gamma(C_2)$ by Observation~\ref{inbetween}, and the corresponding region is 
$$
Z_1(s)=\{(c,d): 0<c\leq \frac{1}{\sqrt{2}},\; 0<d\leq \frac{1}{\sqrt{2}},\; c+d\geq\frac{1}{\sqrt{2}}, \; cd\geq \frac{3}{2\pi^2s} \} \;.
$$

When $d<0$, then $S_{1,2}(C_3)=C(-1+\ii/2,1/2)$ will be mapped in between $\gamma(C_1)$ and $\gamma(C_2)$ by $\gamma$, and the relevant region is
$$
Z_2(s)=\{(c,d): 0<c\leq \frac{1}{\sqrt{2}},\; 0>d\geq -\frac{1}{\sqrt{2}},\; 
c-d\geq\frac{1}{\sqrt{2}},\; cd\leq -\frac{3}{2\pi^2s} \} \;.
$$

Note that $Z_1(s)$ and $Z_2(s)$ are symmetric about the $c$-axis, so we only need to calculate one area.  
 Therefore, by Theorem~\ref{thm:formula for F}, 
$$
F(s)=\frac{2}{3} \cdot 3m(\Omega^{1,2}(s)) = 4m(Z_1(s))\;.
$$
The density function $P(s)$  of the normalized gaps  (the derivative of $F(s)$) is given in Figure~\ref{fig:classicapp3density}.
\begin{figure}[h]
\begin{center}
\includegraphics[width=8cm]
{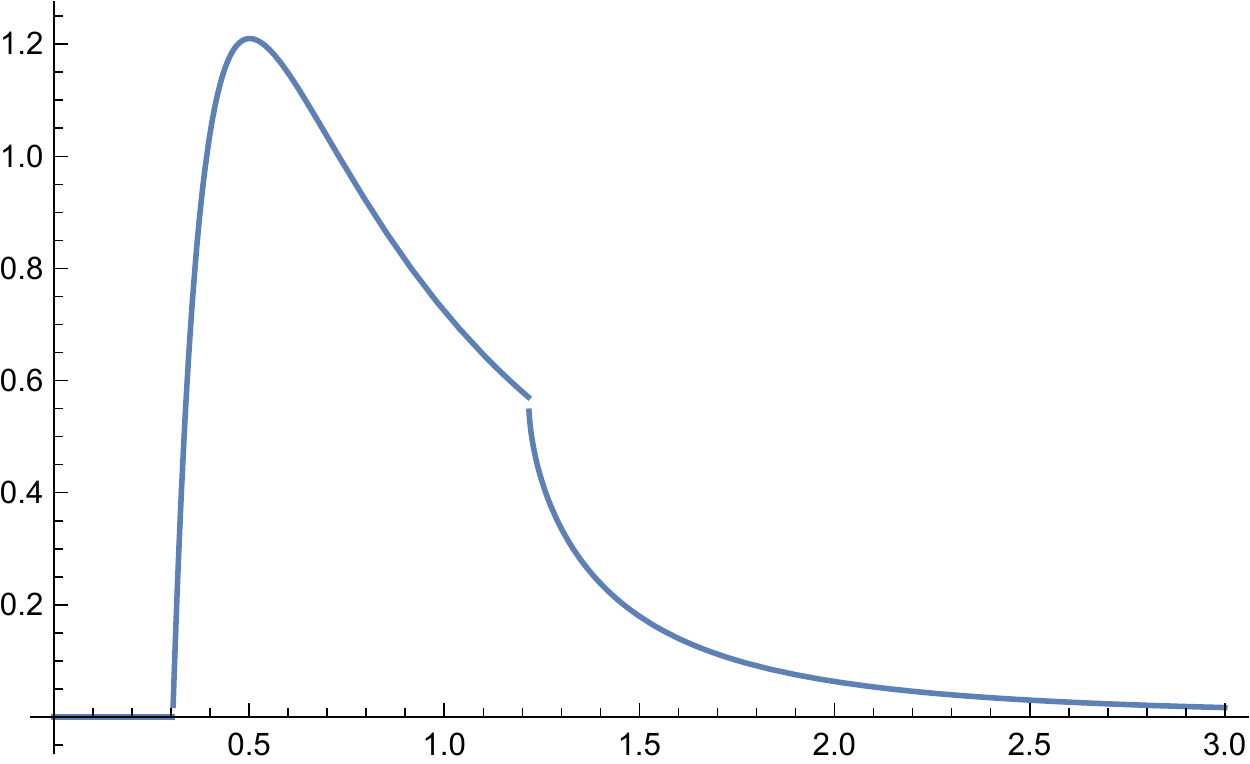}
 \caption{$P(s)$ for the classical Apollonian packing.}
 \label{fig:classicapp3density}
\end{center}
\end{figure}

\subsection{ The Apollonian-$3$ Packing} \label{sec:app3}
Our next example was discovered by Guettler and Mallows \cite{GM}, in which one starts with three mutually tangent circles, and in either of the curvilinear triangles formed by three mutually tangent
circles we pack three more circles, forming a sextuple  $\mathcal K$, see Figure~\ref{ap3ford}.
The base circle is $C_0=\RR$ and the circles of $\mathcal K$ tangent to it are  
$$
C_1=\RR+\bd{i}, \quad C_2=C(\frac{\bd{i}}{2},\frac{1}{2}), \quad C_3=C(\frac{\sqrt{2}}{2}+\frac{1}{4}\ii,\frac{1}{4}), \quad 
C_4=C(\sqrt{2}+\frac{1}{2}\ii,\frac{1}{2})
$$ 
\begin{figure}[h]
\begin{center}
\includegraphics[width=8cm]
{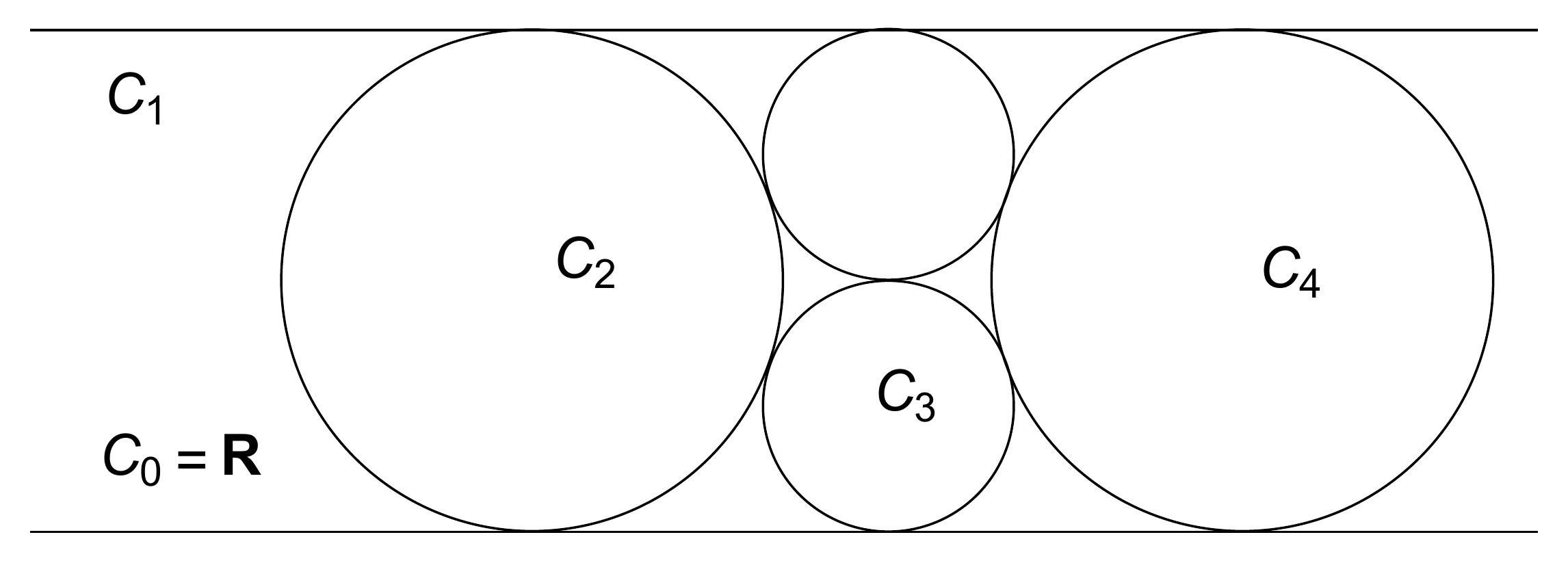}
 \caption{Ford Configuration of Apollonian-3 type}
 \label{ap3ford}
\end{center}
\end{figure}
In this case, $D = 4\sqrt{2}$ and $c_{\PP,C_0}= D/2\pi^2(h-2)=\sqrt{2}/\pi^2$. 

The group $\Gamma_0$ is generated by 
$$z\rightarrow-\bar{z}, \quad z\rightarrow\mat{1&0\\2\sqrt{2}&-1}\bar{z}, \quad 
z\rightarrow\mat{3&-2\sqrt{2}\\2\sqrt{2}&-3}\bar{z}, \quad z\rightarrow\mat{-1&2\sqrt{2}\\0&1}\bar{z}
$$  
Therefore, 
$$
\Gamma=\left\langle\mat{1&2\sqrt{2}\\0&1},\mat{1&0\\2\sqrt{2}&1},\mat{3&-2\sqrt{2}\\-2\sqrt{2}&3}\right\rangle
$$ 

The Hecke group $\Hecke_4$ is generated by $S = \mat{0&-1\\1&0}$ and $T_4 =\mat{1&\sqrt{2}\\0&1}$. It contains $\Gamma$ because 
$$\mat{3&-2\sqrt{2}\\-2\sqrt{2}&3} = T_4^{-1}ST_4^{-2}ST_4^{-1}.
$$

To see that $\Gamma$ is normal in $\Hecke_4$, we check that the conjugates of the generators of $\Gamma$ by the generators $S$, $T_4$ of $\Hecke_4$ still lie in $\Gamma$:   
$$S\mat{1&2\sqrt{2}\\0&1}S^{-1}=\mat{1&0\\2\sqrt{2}&1}^{-1},\quad 
S\mat{1&0\\2\sqrt{2}&1}S^{-1}=\mat{1&2\sqrt{2}\\0&1}^{-1}$$
$$S\mat{3&-2\sqrt{2}\\-2\sqrt{2}&3}S^{-1}=\mat{3&-2\sqrt{2}\\-2\sqrt{2}&3}^{-1}$$
$$T_4\mat{1&2\sqrt{2}\\0&1}T_4^{-1}=\mat{1&2\sqrt{2}\\0&1}$$
$$T_4\mat{1&0\\2\sqrt{2}&1}T_4^{-1}=-\mat{1&2\sqrt{2}\\0&1}\cdot\mat{3&-2\sqrt{2}\\-2\sqrt{2}&3}$$
$$T_4\mat{3&-2\sqrt{2}\\-2\sqrt{2}&3}T_4^{-1}=-\mat{1&0\\2\sqrt{2}&1}\cdot\mat{1&-2\sqrt{2}\\0&1}\;.
$$

There are $6$ (unordered) pairs of circles from $\mathcal{K}_0$:  Four pairs of tangent circles $(C_1,C_2)$, $(C_2,C_3)$, $(C_3,C_4)$, $(C_4,C_1)$,  and two pairs of disjoint circles $(C_1,C_3), (C_2,C_4)$. Each of these cases gives one equivalence class under the action of  $\Hecke_4$:  
$$
T_4(C_1,C_2)=(C_1,C_4), \quad T_4S(C_1,C_4)=(C_3,C_4), \quad 
T_4S(C_3,C_4)=(C_2,C_3),
$$
and
$$
 T_4S(C_1,C_3)=(C_2,C_4)
$$
Hence from formula  \eqref{whatisf} 
$$F(s)=\frac{2}{4\sqrt{2}}(4m(\Omega^{1,2}(s))+4m(\Omega^{1,3}(s)))=\sqrt{2}(m(\Omega^{1,2}(s))+m(\Omega_1^{1,3}(s))$$

 We compute the areas $\Omega_{1}^{1,2}$ and $\Omega_{1}^{1,3}$. Note that if $\gamma=\begin{pmatrix}a&b\\c&d\end{pmatrix}$ then  
\begin{align*}
 \kappa(\gamma(C_1))&=2c^2 , 
 &\kappa(\gamma(C_2)) &=2d^2,\\
 \kappa(\gamma(C_3))&=4(\frac{\sqrt{2}}{2}c+d)^2 ,    
  &\kappa(\gamma(C_4))&=2(\sqrt{2}c+d)^2 \;.
\end{align*}

First we consider $\Omega_{1}^{1,2}$:  The distance between $\gamma(\alpha_1)=\gamma \cdot\infty$ and  $\gamma(\alpha_2)=\gamma \cdot 0$ is $1/|cd|$.  Therefore, we need to consider the cases $d>0$ and $d<0$ separately.  If $d>0$, then the circles $C_3$ and $C_4$ will be mapped by $\gamma$ to some circles in between $\gamma(C_1)$ and $\gamma(C_2)$.  The corresponding region $Z_1$ is
\begin{multline*}
Z_1(s)=\{(c,d):0< \sqrt{2}c\leq1,\; 0< \sqrt{2}d\leq 1, \; \sqrt{2}c+2d\geq 1,\;
\\
 2c+\sqrt{2}d\geq1, \; cd\geq\frac{\sqrt{2}}{\pi^2s}\}\;.
\end{multline*}

If $d<0$, then the circles 
$$S_{1,2}C_3=C(-\frac{\sqrt{2}}{2}+\frac{\ii}{4},\frac{1}{4}) \;\mbox{ and }\; S_{1,2}C_4=C(-\sqrt{2}+\frac{\ii}{2},\frac{1}{2})$$ 
will be mapped in between  $\gamma(C_1)$ and $\gamma(C_2)$.  
The corresponding region $Z_2$ is 
\begin{multline*}
Z_2(s)=\Big\{(c,d):0<{\sqrt{2}}c\leq {1},\; 0>{\sqrt{2}}d\geq {-1}, \;\sqrt{2}c-2d\geq 1,\; \\
2c-\sqrt{2}d\geq1, \;cd\leq-\frac{\sqrt{2}}{\pi^2s}\Big\}\;.
\end{multline*}
The region $\Omega_1^{1,2}(s)$ is the union of $Z_1(s)$ and $Z_2(s)$ which are symmetric, hence the area is 
$$
m(\Omega_1^{1,2}(s)) = 2m(Z_1(s))\;.
$$ 

Next we consider $\Omega_1^{1,3}(s)$. Here  $$\kappa(\gamma(C_3))=(\sqrt{2}c+2d)^2$$ 
and we need to split it into two cases.  If $\sqrt{2}c+2d>0$, then $\gamma C_4$ will be in between $\gamma C_1$ and $\gamma C_3$, and the corresponding region is 
\begin{multline*}
Z_3(s)=
\Big\{(c,d):0<{\sqrt{2}}c\leq {1},\; 0<\sqrt{2}c+2d\leq 1,\; 
2c+\sqrt{2}d\geq 1, \\
\;c(\frac{1}{\sqrt{2}}c+d)\geq\frac{\sqrt{2}}{\pi^2s}
\Big\}\;.
\end{multline*}
If $\sqrt{2}c+2d<0$, then $\gamma C_2$ will be in between $\gamma C_1$ and $\gamma C_3$, and, 
\begin{multline*}
Z_4(s)=
\Big\{(c,d):0<{\sqrt{2}}c\leq {1},\; 0>\sqrt{2}c+2d\geq -1, \;{\sqrt{2}}d\leq -1,\;
\\
 c(\frac{1}{\sqrt{2}}c+d)\leq-\frac{\sqrt{2}}{\pi^2s}\Big\}\;.
\end{multline*}
Then $\Omega_1^{1,3}(s)$ is the union of $Z_3(s)$ and $Z_4(s)$ (which   have equal areas), and 
$$ 
m(\Omega_1^{1,3}(s)) = m(Z_3(s)) + m(Z_4(s))\;.
$$
Thus 
$$
F_3(s) = \sqrt{2}\Big(2m(Z_1(s)) + m(Z_3(s))+m(Z_4(s)) \Big) \;.
$$
The density function  $P(s)$ of the normalized gaps  (the derivative of $F(s)$) is given in Figure~\ref{fig:app3density}.
\begin{figure}[h]
\begin{center}
\includegraphics[width=8cm]
{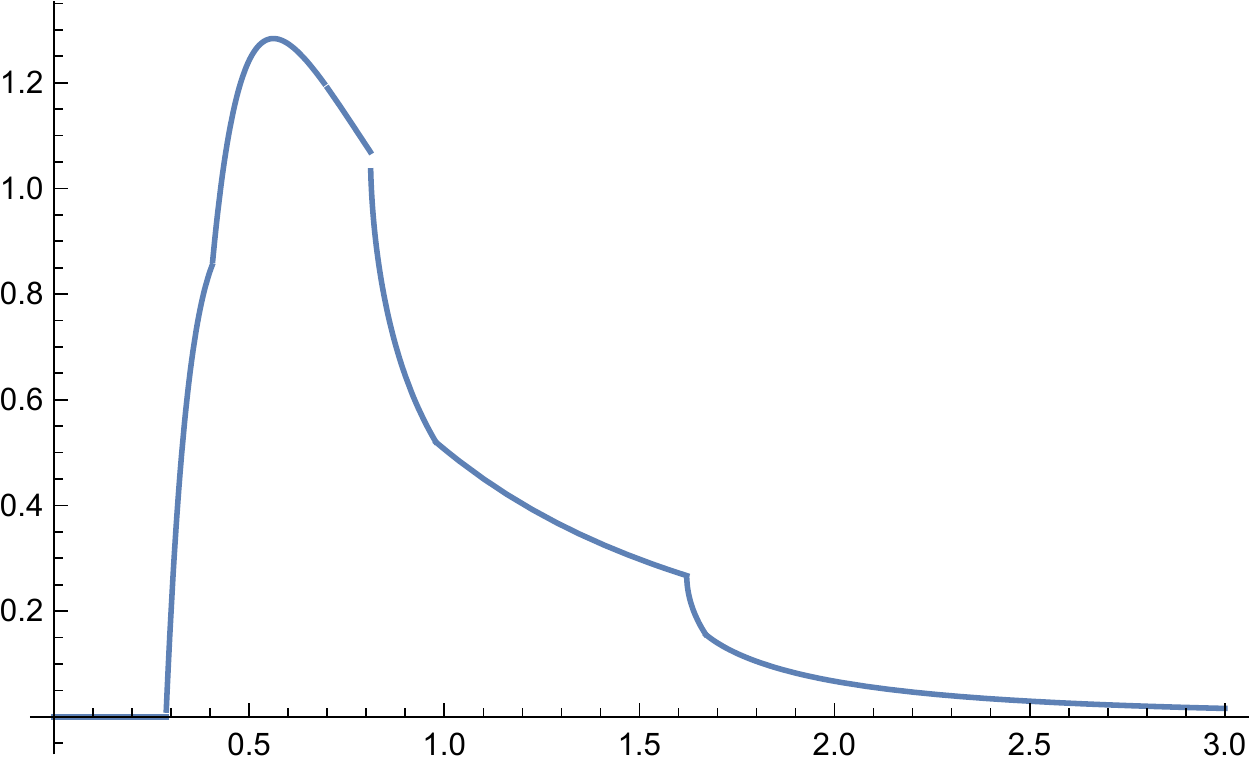}
 \caption{$P(s)$ for the Apollonian-3 packing.}
 \label{fig:app3density}
\end{center}
\end{figure}

\subsection{The Apollonian-9 Packing  } \label{sec:app9}
This packing was introduced in \cite{BGGM}. The initial configuration $\mathcal K$, whose tangency graph is the icosahedron, is shown in  Figure~\ref{ap9ford}. The base circle is $C_0=\RR$ and the circles $\mathcal K_0$ tangent to it are
\begin{multline*}
C_1=\RR+\ii, \quad C_2=C(\frac{\ii}{2},\frac{1}{2}), \quad C_3=C(\frac{\sqrt{5}-1}{2}+\frac{3-\sqrt{5}}{4}\ii,\quad \frac{3-\sqrt{5}}{4}), \\
C_4=C(1+\frac{3-\sqrt{5}}{4}\ii,\quad 
\frac{3-\sqrt{5}}{4}), \quad 
C_5=C(\frac{\sqrt{5}+1}{2}+\frac{\ii}{2}, \frac{1}{2})
\end{multline*}
  The constant $c_{\PP,C_0}$ in this case is $\frac{5(\sqrt{5}+1)}{12\pi^2}$. 
\begin{figure}[h]
\begin{center}
\includegraphics[width=8cm]
{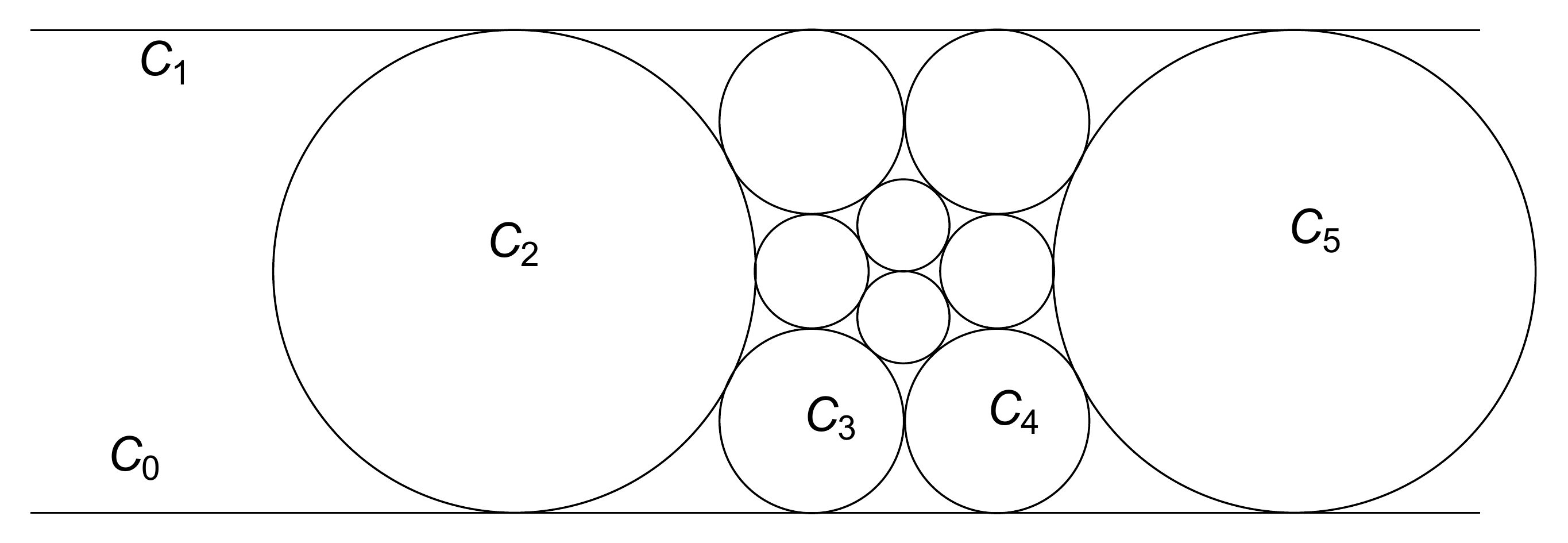}
 \caption{Ford Configuration of Apollonian-9 type}
 \label{ap9ford}
\end{center}
\end{figure}


The group  $\Gamma_0$ is generated by the reflections 
$$z\rightarrow \mat{-1&0\\0&1}\bar{z}, \quad z\rightarrow \mat{1&0\\\sqrt{5}+1&-1}\bar{z}, \quad  z\rightarrow \mat{2+\sqrt{5}&-1-\sqrt{5}\\\sqrt{5}+3&-2-\sqrt{5}}\bar{z},$$ $$z\rightarrow \mat{2+\sqrt{5}&-3-\sqrt{5}\\\sqrt{5}+1&-2-\sqrt{5}}\bar{z}, \quad z\rightarrow \mat{-1&1+\sqrt{5}\\0&1}\bar{z}\;.
$$  
Taking the products of the first matrix $z\mapsto -\bar z$ and each of the other four matrices, we get generators for $\Gamma$: 
 $$\Gamma=\left\langle\mat{1&1+\sqrt{5}\\0&1},\mat{1&0\\1+\sqrt{5}&1},\mat{2+\sqrt{5}&3+\sqrt{5}\\1+\sqrt{5}&2+\sqrt{5}},\mat{2+\sqrt{5}&1+\sqrt{5}\\3+\sqrt{5}&2+\sqrt{5}}\right\rangle.$$  

 We claim that $\Gamma\subset  \Hecke_5$, the Hecke-5 group. 
which is generated by 
$$S=\mat{0&-1\\1&0} \quad \mbox{and} \quad T_5=\mat{1&\frac{\sqrt{5}+1}{2}\\0&1}\;,
$$  
We check that each generator of $\Gamma$ can be written as word of $S$ and $T_5$ and hence lies in $\Hecke_5$:
Clearly, 
$$
\mat{1&1+\sqrt{5}\\0&1}=T_5^{2},\quad 
\mat{1&0\\1+\sqrt{5}&1}=ST_5^{-2}S^{-1}
$$
Moreover 
$$
\mat{2+\sqrt{5}&3+\sqrt{5}\\ \sqrt{5}+1&2+\sqrt{5}}=T_5ST_5^{2}ST_5
\quad 
$$
and
$$
\mat{2+\sqrt{5}&1+\sqrt{5}\\\sqrt{5}+3&2+\sqrt{5}}
=ST_5^{-1}ST_5^{-2}ST_5^{-1}S^{-1}
$$
as claimed.

Since $\Gamma$ is generated by elements of $\Hecke_5$, it is a subgroup of that group, and since both groups have finite volume, 
$\Gamma$ is a subgroup of finite index in $\Hecke_5$, and in particular   is non-arithmetic.

To see that $\Gamma$ is normal in $\Hecke_5$, we check that the conjugates of the generators of $\Gamma$ by the generators $S$, $T_5$ of $\Hecke_5$ still lie in $\Gamma$:   
 $$
S\mat{1&1+\sqrt{5}\\0&1}S^{-1}=\mat{1&0\\1+\sqrt{5}&1}^{-1},\quad S\mat{1&0\\1+\sqrt{5}&1}S^{-1}=\mat{1&1+\sqrt{5}\\0&1}^{-1};
$$
$$
S\mat{2+\sqrt{5}&3+\sqrt{5}\\1+\sqrt{5}&2+\sqrt{5}}S^{-1}=\mat{2+\sqrt{5}&1+\sqrt{5}\\3+\sqrt{5}&2+\sqrt{5}}^{-1};
$$
$$
S\mat{2+\sqrt{5}&1+\sqrt{5}\\3+\sqrt{5}&2+\sqrt{5}}S^{-1}=\mat{2+\sqrt{5}&3+\sqrt{5}\\1+\sqrt{5}&2+\sqrt{5}}^{-1};
$$
$$
T_5\mat{1&1+\sqrt{5}\\0&1}T_5^{-1}=\mat{1&1+\sqrt{5}\\0&1};
$$
$$T_5\mat{1&0\\1+\sqrt{5}&1}T_5^{-1}=
-\mat{1&1+\sqrt{5}\\0&1} \mat{2+\sqrt{5}&3+\sqrt{5}\\1+\sqrt{5}&2+\sqrt{5}}^{-1}
$$
 and hence 
$$
T_5\mat{2+\sqrt{5}&3+\sqrt{5}\\1+\sqrt{5}&2+\sqrt{5}}T_5^{-1}=-\mat{1&1+\sqrt{5}\\0&1}\mat{1&0\\1+\sqrt{5}&1}^{-1};
$$
$$
T_5\mat{2+\sqrt{5}&1+\sqrt{5}\\3+\sqrt{5}&2+\sqrt{5}}T_5^{-1}=
-\mat{1&1+\sqrt{5}\\0&1}\mat{2+\sqrt{5}&1+\sqrt{5}\\3+\sqrt{5}&2+\sqrt{5}}^{-1}.
$$

 The pairs $(C_1,C_2)$, $(C_1,C_3)$ are representatives of two conformally equivalent classes of pairs of circles from $\mathcal{K}_0$, characterized by the two circles being  tangent or disjoint and $\Hecke_5$ acts transitively on each class:  Indeed, 
$$
(C_1,C_2)\xrightarrow{T_5}(C_1,C_5)\xrightarrow{T_5S}(C_4,C_5)\xrightarrow{T_5S}(C_3,C_4)\xrightarrow{T_5S}(C_2,C_3)
$$
and 
$$
(C_1,C_3)\xrightarrow{T_5S}(C_2,C_5)\xrightarrow{T_5S}(C_1,C_4)\xrightarrow{T_5S}(C_3,C_5)\xrightarrow{T_5S}(C_2,C_4)
$$
Each equivalence class contains 5 pairs of circles.   Therefore,  
\begin{equation*}
F(s)=\frac{4}{5(\sqrt{5}+1)}\left(5m(\Omega_1^{1,2}(s))+10m(\Omega_1^{1,3}(s))\right) 
\end{equation*}

  The curvatures of the transformed circles are: 
  \begin{equation*}
  \begin{split}
  \kappa(\gamma (C_1))=2c^2,\quad \kappa(\gamma (C_2))=2d^2,\quad \kappa(\gamma (C_3))=2(c+\frac{\sqrt{5}+1}{2}d)^2,
  \\
  \kappa(\gamma (C_4))=2(\frac{\sqrt{5}+1}{2}c+\frac{\sqrt{5}+1}{2}d)^2,
  \quad  \kappa(\gamma (C_5))=2(\frac{\sqrt{5}+1}{2}c+d)^2\;.
  \end{split}
  \end{equation*}
Note that $\kappa(\gamma C_4)$ is greater than   $\kappa(\gamma  C_3)$ and $\kappa(\gamma  C_5)$.

We first compute $\Omega_1^{1,2}(s)$.   
If $d>0$, then by Observation~\ref{inbetween}, the circles $C_3,C_4,C_5$ will mapped by $\gamma$ in between $\gamma(C_1)$ and $\gamma(C_2)$.  The corresponding region is, using \eqref{bounding5infty}, 
\begin{multline*}
Z_1=\Big\{(c,d):0<c\leq\frac{\sqrt{2}}{2}, \; 0<d\leq \frac{\sqrt{2}}{2}, \; c+\frac{\sqrt{5}+1}{2}d\geq \frac{\sqrt{2}}{2},\;
\\
 \frac{\sqrt{5}+1}{2}c+d\geq \frac{\sqrt{2}}{2},\; cd\geq\frac{5(\sqrt{5}+1)}{12\pi^2s}\Big\}
\;.
\end{multline*}
The condition $\kappa(\gamma(C_4))\geq 1$ is redundant here. 

If $d<0$, then 
\begin{equation*}
\begin{split}
S_{1,2}(C_3)&=C(\frac{-\sqrt{5}+1}{2}+\frac{3-\sqrt{5}}{4}\ii, \frac{3-\sqrt{5}}{4}),
\\ 
S_{1,2}(C_4)&=C(-1+\frac{3-\sqrt{5}}{4}\ii,\frac{3-\sqrt{5}}{4})\;,\\ 
S_{1,2}(C_5)&=C(\frac{-\sqrt{5}-1}{2}+\frac{\ii}{2}, \frac{1}{2})
\end{split}
\end{equation*}
will be mapped in between $\gamma(C_1)$ and $\gamma(C_2)$.  The corresponding region is, by \eqref{bounding2infty}, 
\begin{multline*}
Z_2=\Big\{(c,d):0<c\leq\frac{\sqrt{2}}{2}, \; 0>d\geq-\frac{\sqrt{2}}{2}, \; 
 c-\frac{\sqrt{5}+1}{2}d\geq \frac{\sqrt{2}}{2}, \; \\ 
\frac{\sqrt{5}+1}{2}c-d\geq \frac{\sqrt{2}}{2}, \;  
 -cd\leq-\frac{5(\sqrt{5}+1)}{12\pi^2s}\Big\}\;.
\end{multline*}
The region $\Omega_1^{1,2}(s)$ is the union of $Z_1$ and $Z_2$, which are symmetric.

Now  consider $\Omega_1^{1,3}(s)$.  By \eqref{bounding6infty}, if $c+\frac{\sqrt{5}+1}{2}d>0$, then  the images $\gamma C_4$ and $\gamma C_5$ will be between $\gamma C_1$ and $\gamma C_3$, and the corresponding region is 
\begin{multline*}
Z_3= \{ (c,d):0<c\leq\frac{\sqrt{2}}{2}, \; 0<c+\frac{\sqrt{5}+1}{2}d\leq \frac{\sqrt{2}}{2}, \; 
\frac{\sqrt{5}+1}{2}c+d\geq \frac{\sqrt{2}}{2},\;\\
 \frac{\sqrt{5}+1}{2}c+\frac{\sqrt{5}+1}{2}d\geq \frac{\sqrt{2}}{2},\; c(\frac{\sqrt{5}-1}{2}c+d)\geq \frac{5(\sqrt{5}+1)}{12\pi^2s} \}
\end{multline*}

Similarly, if $c+\frac{\sqrt{5}+1}{2}d<0$, then $\gamma C_2$ will be in between $\gamma C_1$ and $\gamma C_3$, and the region is
\begin{multline*}
Z_4=\{(c,d):0<c\leq\frac{\sqrt{2}}{2}, \; 0>c+\frac{\sqrt{5}+1}{2}d\geq -\frac{\sqrt{2}}{2}, \; d\leq -\frac{\sqrt{2}}{2}, \;
\\ c(\frac{\sqrt{5}-1}{2}c+d)\leq -\frac{5(\sqrt{5}+1)}{12\pi^2s}\}
\end{multline*}
The region  $\Omega_1^{1,3}(s)$ is the union of $Z_3$ and $Z_4$. 
Therefore,  
\begin{equation*}
\begin{split}
F(s)&=\frac{4}{5(\sqrt{5}+1)}\left(5m(\Omega_1^{1,2}(s))+10m(\Omega_1^{1,3}(s))\right)\\
& = \frac{8}{\sqrt{5}+1}\Big(m(Z_1(s)) + m(Z_3(s)) + m(Z_4(s)) \Big)\;.
\end{split}
\end{equation*}
The density function    of the normalized gaps 
is given in Figure~\ref{fig:app9density}.
\begin{figure}[b]
\begin{center}
\includegraphics[width=8cm]
{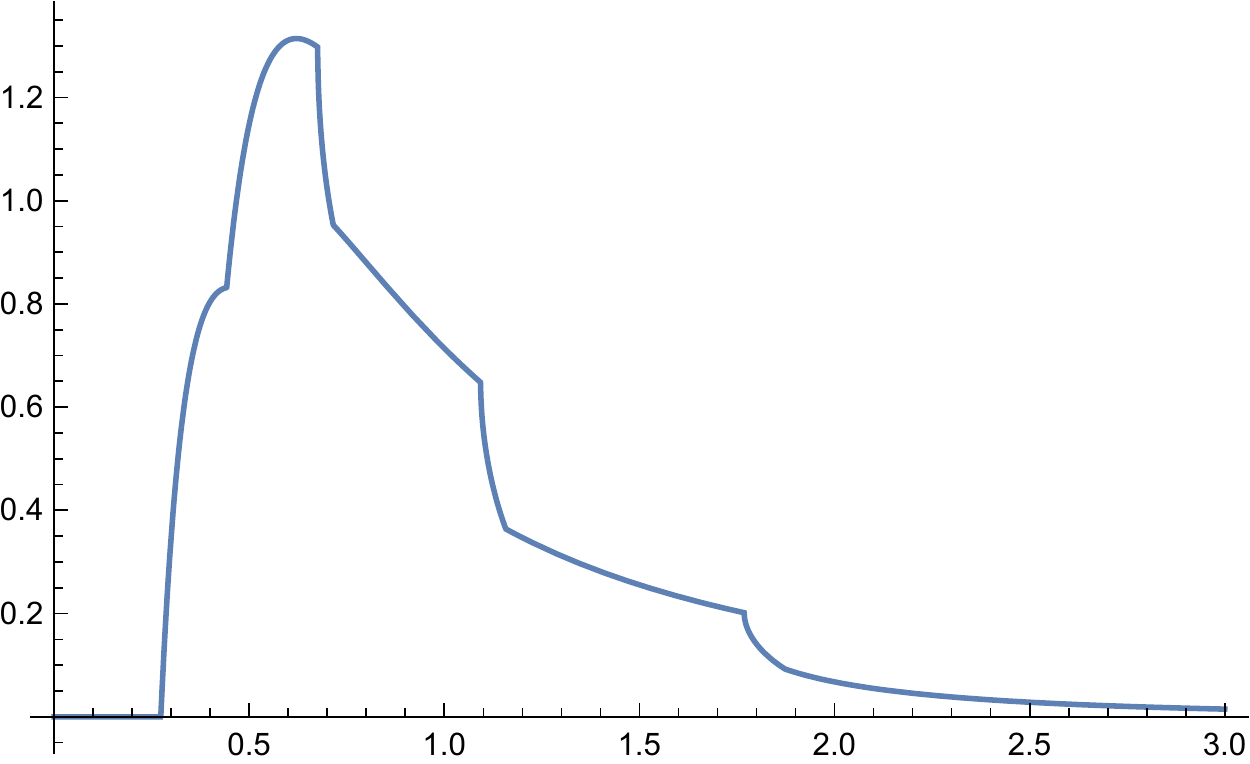}
 \caption{$P(s)=F'(s)$ for the Apollonian-9 packing.}
 \label{fig:app9density}
\end{center}
\end{figure}

\end{document}